\documentclass[11pt]{amsart}

\usepackage{epigamath}

\usepackage[english]{babel}

\numberwithin{equation}{section}

\newtheorem{theo}{Theorem}[section]

\newtheorem{prop}[theo]{Proposition}

\newtheorem{lemma}[theo]{Lemma}
\newtheorem{cor}[theo]{Corollary}
\newtheorem{corollary}[theo]{Corollary}

\theoremstyle{definition}
\newtheorem{defi}[theo]{Definition}
\newtheorem{definition}[theo]{Definition}
\newtheorem{art}[theo]{}

\theoremstyle{remark}
\newtheorem{rem}[theo]{Remark}

\newtheorem{example}[theo]{Example}

\newcommand{\Fcal}{{\mathcal{F}}}

\newcommand{\Gcal}{{\mathcal{G}}}

\newcommand{\Hscr}{{\mathscr{H}}}
\newcommand{\Hcal}{{\mathcal{H}}}

\newcommand{\Lcal}{{\mathscr L}}

\newcommand{\Mcal}{{\mathscr M}}
\newcommand{\Nbb}{{\mathbb{N}}}
\newcommand{\Ncal}{{\mathcal{N}}}

\newcommand{\Ocal}{{\mathcal{O}}}
\newcommand{\Pbb}{{\mathbb{P}}}

\newcommand{\Ucal}{{\mathcal{U}}}

\newcommand{\Vcal}{{\mathcal{V}}}

\newcommand{\Xcal}{{\mathscr X}}

\newcommand{\an}{{\mathrm{an}}}

\newcommand{\N}{\mathbb{N}}

\newcommand{\Q}{\mathbb{Q}}

\newcommand{\R}{\mathbb{R}}

\newcommand{\Z}{\mathbb{Z}}

\newcommand{\mb}{{\mathbf m}}

\renewcommand{\and}{\operatorname{and}}

\DeclareMathOperator{\Coker}{Coker}

\renewcommand{\div}{{\operatorname{div}}}

\DeclareMathOperator{\en}{E}

\DeclareMathOperator{\Env}{{P}}

\DeclareMathOperator{\Ker}{Ker}

\DeclareMathOperator{\Pic}{{Pic}}

\DeclareMathOperator{\Spec}{{Spec}}

\DeclareMathOperator{\supp}{{supp}}

\DeclareMathOperator{\vol}{ {vol}}
\DeclareMathOperator{\metr}{{\|\cdot \|}}
\DeclareMathOperator{\abs}{{|\cdot |}}
\DeclareMathOperator{\KL}{{\mathscr L}}
\DeclareMathOperator{\KM}{{\mathscr M}}

\DeclareMathOperator{\KX}{{{\mathscr X}}}

\DeclareMathOperator{\length}{c}

\newcommand{\Lan}{{L^{\an}}}
\newcommand{\Xan}{{X^{\an}}}

\newcommand{\Ko}{{\ensuremath{K^\circ}}}
\newcommand{\Koo}{{K^{\circ\circ}}}
\newcommand{\Kt}{{\tilde{K}}}

\newcommand{\st}{ \ \big| \ }

\newcommand{\hh}{\hat{h}}

\newcommand{\cO}{\mathcal{O}}

\newcommand{\Uan}{U^{\mathrm{an}}}

\renewcommand{\d}{\delta}
\newcommand{\e}{\varepsilon}

\newcommand{\p}{\psi}

\DeclareMathOperator{\Hom}{Hom}

\newcommand{\redu}{\operatorname{red}}

\renewcommand{\div}{\mathrm{div}}

\EpigaVolumeYear{5}{2021} \EpigaArticleNr{13} \ReceivedOn{November 16,
2020}
\InFinalFormOn{May 25, 2021}
\AcceptedOn{August 18, 2021}

\title{Non-Archimedean volumes of metrized nef line bundles}
\titlemark{Non-Archimedean volumes of metrized nef line bundles}

\author{S\'ebastien Boucksom}
\address{Centre de Math\'ematiques Laurent Schwartz, Ecole Polytechnique and CNRS, Institut Polytechnique de Paris}
\email{sebastien.boucksom@polytechnique.edu}

\author{Walter Gubler}
\address{W. Gubler, Mathematik, Universit{\"a}t 
	Regensburg, 93040 Regensburg, Germany} 
\email{walter.gubler@mathematik.uni-regensburg.de}

\author{Florent Martin}
\address{F. Martin, Mathematik, Universit{\"a}t 
	Regensburg, 93040 Regensburg, Germany}
\email{florent.guy.martin@gmail.com}

\authormark{S.~Boucksom, W.~Gubler, and F.~Martin}

\AbstractInEnglish{\scriptsize{Let $L$ be a line bundle on a proper, geometrically reduced scheme $X$ over a non-trivially valued non-Archimedean field $K$. Roughly speaking, the non-Archimedean volume {of a continuous metric on the Berkovich analytification of $L$} measures the asymptotic growth of the space of small sections  of tensor powers of $L$.  For a continuous semipositive metric on $L$ {in the sense of Zhang}, we show first that the non-Archimedean volume agrees with the energy. The existence of such a semipositive metric yields  that $L$ is nef. 
	A second result is that the non-Archimedean volume is differentiable at any semipositive continuous metric. {These results are known when $L$ is ample, and the purpose of this paper is to generalize them to the nef case}. The method is based on a detailed study of the content and the volume of a finitely presented torsion module over the (possibly non-Noetherian) valuation ring of $K$.}}



\TitleInFrench{Volumes non-archim\'ediens des fibr\'es en droites nef m\'etris\'es}

\AbstractInFrench{\scriptsize{Soit $L$ un fibr\'e en droites sur un sch\'ema propre et g\'eom\'etriquement r\'eduit $X$ d\'efini sur un corps $K$ muni d'une valuation non-triviale non-archim\'edienne. \emph{Grosso modo}, le volume non-archim\'edien d'une m\'etrique continue sur  l'analytifi\'e de Berkovich de $L$ mesure la croissance asymptotique de l'espace des petites sections des puissances tensorielles de $L$. Pour une m\'etrique continue semi-positive sur $L$ au sens de Zhang, nous montrons tout d'abord que le volume non-archim\'edien co\"incide avec l'\'en\'ergie. L'existence d'une telle m\'etrique semi-positive impose le caract\`ere nef de $L$. La diff\'erentiabilit\'e du volume non-archim\'edien en toute m\'etrique continue semi-positive constitue un deuxi\`eme r\'esultat. Ces r\'esultats sont connus lorsque $L$ est ample et l'objectif de cet article est de les g\'en\'eraliser au cas nef. La m\'ethode s'appuie sur une \'etude d\'etaill\'ee du contenu et du volume d'un module de torsion de pr\'esentation finie sur l'anneau (\'eventuellement non noeth\'erien) de valuation de $K$.}}

\acknowledgement{S\'ebastien Boucksom was partly supported by the ANR project GRACK. Walter Gubler and Florent Martin were supported by the collaborative research 
	center SFB~1085 funded by the Deutsche Forschungsgemeinschaft.}

\begin{document}


\removeabove{0,6cm}
\removebetween{0,6cm}
\removebelow{0,6cm}

\maketitle

\begin{prelims}

\DisplayAbstractInEnglish

\bigskip

\DisplayKeyWords

\medskip

\DisplayMSCclass

\bigskip

\languagesection{Fran\c{c}ais}

\bigskip

\DisplayTitleInFrench

\medskip

\DisplayAbstractInFrench

\end{prelims}


\newpage

\setcounter{tocdepth}{1}

\tableofcontents


\section{Introduction}

The volume is an important invariant in algebraic geometry measuring asymptotically the size of the space of global sections. For a line bundle $L$ on an $n$-dimensional projective variety $Y$ over an algebraically closed field $k$, it is given by 
\begin{equation*}
\label{eq volume intro}
\vol(L) \coloneqq \limsup_{m \to \infty} \frac{h^0(Y,mL)}{ m^n /n!}.
\end{equation*}
Here and in the following, we use additive notation for line bundles and metrics. The volume has many nice properties like continuity and differentiability for which we refer to Lazersfeld's books~\cite{Laz1} and~\cite{Laz2}. In Arakelov geometry, there is a similar invariant called the arithmetic volume. It measures asymptotically the size of the number of small sections. For a projective arithmetic variety $\Xcal$ of relative dimension $n$  over the ring of integers $\Ocal_F$ of a number field $F$ and a  line bundle $\Lcal$ on $\Xcal$ endowed with Hermitian metrics $\phi_v$ for each Archimedean place $v$, it is defined by 
\begin{equation*}
\widehat{\vol}(\overline L)\coloneqq \limsup_{m \to \infty}
\frac{\log \# \{s\in H^0(\Xcal,m\Lcal)\mid  
	\|s\|_{m\phi_v}\le 1,\
	\forall v|\infty\}}{ m^{n+1} /(n+1)!}
\end{equation*}
where $\|s\|_{m\phi_v}$ is the supremum norm of $s$ associated to the metric $m\phi_{v}$. The  arithmetic volume behaves similarly to the classical volume above, see Chen~\cite{chen08:_posit}, Moriwaki~\cite{moriwaki-2009} and Yuan~\cite{Yua08}. In the philosophy of Lang, N\'eron and Weil, arithmetic invariants are always influenced by local invariants depending only on a single place of the number field $F$. 
The main object of study of this paper is a local variant of the arithmetic volume which we study over any non-Archimedean field. For its relation to the global arithmetic volume in case of a number field, we refer to~\cite[Remark 4.1.7]{BGJKM}.

In the following, $K$ is a non-Archimedean field, \textit{i.e.} a field $K$ endowed with a  complete non-Archimedean absolute value $\abs$, assumed to be \textbf{non-trivially valued}, here and throughout the paper. The valuation ring of $K$ is denoted by $\Ko$, and the residue field by $\Kt$.  We consider a line bundle $L$ on an $n$-dimensional reduced proper scheme  $X$ over $K$. Then the Berkovich analytification $\Xan$ is compact and we consider a  metric $\phi$ on $L$ which is continuous with respect to the Berkovich topology. 
We denote the associated {\it supremum norm} on $H^0(X,L)$ by $\metr_\phi$. Using~\cite[\S 2.1]{BE}, we note that $\metr_\phi$ induces a canonical norm $\det(\metr_\phi)$ on the one-dimensional $K$-vector space $\det(H^0(X,L))$. Since a norm on a one-dimensional vector space is unique up to scaling, for any other continuous metric $\psi$ of $L$, we get a well-defined positive number $\det(\metr_\phi)/\det(\metr_\psi)$. Following~\cite[\S 2.3, \S 9.2]{BE}, we define our local  volume by 
\begin{equation} \label{def na volume}
\vol(L,\phi,\psi) \coloneqq \limsup_{m \to  \infty}  \frac {n!}  {m^{n+1}} \cdot
\log \left( \det(\metr_{m\psi})/\det(\metr_{m\phi}) \right).
\end{equation}
We call it here the {\it non-Archimedean volume of $L$ with respect to $(\phi,\psi)$}. In contrast to the global case, it is a relative notion which is only well-defined with respect to a pair of metrics $(\phi,\psi)$. For more details, we refer to \S~\ref{subsection comparison volumes}. Non-Archimedean volumes were first introduced by Kontsevich and Tschinkel in~\cite{kontsevitch-tschinkel} and they proposed their differentiability.  The main result of this paper will show that this is true over any non-Archimedean field. Chen and Maclean~\cite{chen-maclean} studied a variant of a local volume and their results yield that the $\limsup$ in \eqref{def na volume} is in fact a limit if $X$ is geometrically reduced (see~\cite[Theorem 9.8]{BE}).

To describe our first result, we recall from algebraic geometry that the volume of a nef line bundle $L$ on the proper scheme $Y$ over $k$ is the degree $\deg_L(X)$. We are looking for a similar result in case of non-Archimedean volumes of the line bundle $L$ on an $n$-dimensional geometrically reduced proper scheme $X$ over any non-Archimedean field $K$. Let $\phi_1,\phi_2$ be continuous metrics of the line bundle $L$ which are semipositive in the sense of Zhang ~\cite{zhang-95}. We recall semipositivity in \S~\ref{subsection model and metrics} and \S~\ref{subsection MA measures}. 
The analogue of the degree in the relative setting is the {\it energy}
$$ 
\en(L, \phi_1,\phi_2) 
\coloneqq
\frac{1}{n+1} 
\sum_{j=0}^n 
\int_{X^\an}
(\phi_1-\phi_2)
(dd^c \phi_1)^{ j} \wedge (dd^c \phi_2)^{n-j}.
$$
On the right-hand side, we use the Monge--Amp\`ere measures $(dd^c \phi_1)^{ j} \wedge (dd^c \phi_2)^{n-j}$ on the Berkovich space $\Xan$ introduced by Chambert--Loir~\cite{chambert-loir-2006}. We refer to \S~\ref{subsection MA measures} for details about these Radon measures including a proof of  locality principle. The energy was introduced in non-Archimedean geometry in~\cite{BFJ2}, we recall the basic properties in \S~\ref{subsection energy}. 

\begin{theo} \label{theo energy intro} Let $L$ be a line bundle on a reduced proper scheme $X$ over a non-Archimedean 
	field $K$. If $\phi_1, \phi_2$ are two continuous semipositive metrics on $L^{\an}$, 
	then 
	\begin{equation*}
	\vol(L,\phi_1, \phi_2) = \en(L,\phi_1, \phi_2).
	\end{equation*}
\end{theo}

In the corresponding Archimedean situation and for $L$ (nef and) big, this was shown in~\cite[Theorem~A]{BB10}. For $K$ discretely valued, the theorem was proved in~\cite[Theorem A]{BGJKM}. For $L$ ample and $K$ any non-Archimedean field, this is a result given in~\cite[Theorem A]{BE}. We generalize it here in Theorem~\ref{cor3 energy} to any (nef) line bundle in case of an arbitrary non-Archimedean field $K$.

The main result of our paper is the following differentiability of the non-Archimedean volume.
\begin{theo}  \label{theo differentiability intro}
	Let $L$ be a line bundle on the $n$-dimensional proper, geometrically reduced scheme $X$ over a non-Archimedean field $K$. 
	Let $\phi$ be a continuous semipositive metric on  $L$ and let $f\colon X^{\an} \to \R$ be continuous. 
	Then $\vol(L,\phi+tf,\phi)$ is 
	differentiable at $t=0$ and 
	\begin{equation} \label{equation differentiability}
	\frac{d}{dt}\bigg|_{t=0} \vol(L,\phi+tf,\phi) =
	\int_{X^{\an}} f \, (dd^c \phi)^{n}.
	\end{equation}
\end{theo}

This formula is the non-Archimedean analogue of~\cite[Theorem~B]{BB10}, 
and was proposed 
by  Kontsevich and
Tschinkel~\cite[\S 7.2]{kontsevitch-tschinkel}. Differentiability of arithmetic volumes was proven by Yuan (see~\cite{Yua08} and~\cite[\S 4.4]{chen-lms}). 
In the case of a discretely valued field $K$, Theorem~\ref{theo differentiability intro} was shown in~\cite[Theorem B]{BGJKM}. For an ample line bundle, this was generalized in~\cite[Theorem A]{BGM}
for any non-Archimedean base field $K$. We will deduce from it the more general Theorem~\ref{theo differentiability intro} using the additional tools described below. This will be done in 
Theorem~\ref{corollary differentiablity yuan's argument}.

The proofs of Theorem~\ref{theo energy intro} 
and Theorem~\ref{theo differentiability intro} are similar, but for the latter 
additional problems arise from leaving the nef cone. 
Our arguments were inspired by the techniques for the proofs of the arithmetic 
Hilbert--Samuel theorem by  Abbes and  Bouche 
\cite{abbes-bouche} and  of a general equidistribution result by  Yuan~\cite{Yua08}. Our proofs follow the overall plan in~\cite{BGJKM} for the same statements in the DVR case, but we have to adapt it here at several places to deal with the non-Noetherian situation.

A crucial tool in the proofs of Theorem~\ref{theo energy intro} 
and Theorem~\ref{theo differentiability intro} is the \emph{volume} of a line bundle $L$ over an $n$-dimensional finitely presented projective torsion scheme $Y$ over $\Ko$ which we will introduce in Section~\ref{section volume torsion}. It follows from the direct image theorem (see Ullrich~\cite[Theorem 3.5]{ullrich95}) that $H^q(Y,L)$ is a finitely presented torsion $\Ko$-module $M$ and  we define $h^q(Y,L)\coloneqq\length(M)$. Here, $\length(M)$ is the \emph{content} of a finitely presented torsion $\Ko$-module $M$  which is a generalization of the length to our non-Noetherian situation and  which was already considered by Scholze~\cite{Sch13}, Temkin~\cite{temkin-metr} and in~\cite{BE} (see \S~\ref{lattices and content} for details). We will see in \S~\ref{Hilbert-Samuel theory} that the invariants $h^q(Y,L)$ share many properties of the usual Hilbert--Samuel theory. Influenced by a similar construction in algebraic geometry by K\"uronya~\cite{Kur06}, we define the {\it $q$-th asymptotic cohomological functions} by
$$\hat{h}^q(Y,L) \coloneqq \limsup_{m \to \infty} \frac{h^q(Y,mL)}{m^n/n!}.$$
For $q=0$, we call it the {\it volume of $L$} and we set $\vol(Y,L) \coloneqq\hat{h}^0(Y,L)$. In Section~\ref{section volume torsion}, we show the asymptotic cohomological functions are  continuous and homogeneous of degree $n$.

Another basic ingredient is the following \emph{Hilbert--Samuel type formula}: let $\Xcal$ be a projective flat scheme over $\Ko$ with generic fiber $X \coloneqq \Xcal \otimes_\Ko K$ and let $n \coloneqq \dim(X)$. Let $\Lcal$ be a nef line bundle on $\Xcal$  with associated model metric $\phi_\Lcal$ on $L \coloneqq \Lcal|_X$. Let $D$ be a vertical effective Cartier divisor on $\Xcal$ with associated model {function}  $\phi_D \coloneqq \phi_{\Ocal(D)}$. Then we have
\begin{equation} \label{HS formula for volume}
\vol(D,\Lcal)= \int_\Xan \phi_D \, (dd^c \phi_\Lcal)^n.
\end{equation}
Based on crucial results for the Deligne pairing in~\cite{BE}, this formula was proven in~\cite[Theorem 2.4]{BGM} for $\Lcal$ ample. In Proposition~\ref{volume formula for real nef}, the continuity of the volume on the left-hand side allows us to generalize  \eqref{HS formula for volume} for nef line bundles $\Lcal$. 
Then both Theorem~\ref{theo energy intro} 
and Theorem~\ref{theo differentiability intro} follow from  \eqref{HS formula for volume} by using a variant of Yuan's filtration argument.

In Section~\ref{Section: differentiability of energy}, we give two applications of the above theorems. We suppose that $L$ is a line bundle on a proper geometrically reduced scheme $X$ over $K$. We assume  $\psi$ is a continuous metric on $L$ such that the \emph{semipositive envelope} 
$$\Env(\psi) := \sup \{ \phi \st \phi \text{ is a continuous semipositive metric on $\Lan$ and } \phi \leq \psi\}$$
is a  continuous metric on $L$. This is expected to hold for all normal projective varieties and semiample line bundles (see~\cite[Conjecture 7.31]{BE}). Then we show in Corollary~\ref{lemma volume energy} that 
\begin{equation} \label{energy=volume}
\vol(L,\phi,\psi) = \en(L,\Env(\phi),\Env(\psi))
\end{equation}
where $\psi$ is another continuous metric on $L$ with $\Env(\psi)$ continuous. The formula \eqref{energy=volume} was shown in~\cite[Corollary~9.16]{BE} for $L$ ample. For $\phi$ as above, we will prove in Theorem~\ref{lemma orthogonality} the \emph{orthogonality property}
\begin{equation} \label{orthogonality property-intro}
\int_\Xan (\Env(\phi) - \phi) \,(dd^c\Env(\phi))^n = 0.
\end{equation}
This is a crucial property in the proof of the existence of solutions of non-Archimedean Monge--Amp\`ere equations, see~\cite{BFJ2}.	

We repeat here that we assume in the whole paper that the absolute value of the non-Archimedean field $K$ is non-trivial. The reason is that we want to work with semipositive metrics in the sense of Zhang. Note that in the trivially valued case, every line bundle on a proper variety has only one model metric and hence there is also a unique semipositive metric. To get a rich theory in the trivially valued case, one has to use a different semipositivity notion as for example in \cite{BE} or \cite{BJ18}. Still, our results can also be applied in the trivially valued case using that the Monge--Amp\`ere measures, the energy and the non-Archimedean volume are compatible with base change (see \cite{BE}).

\subsection*{Acknowledgments}

We are  grateful to Jos\'e Burgos Gil, Antoine Ducros, Dennis Eriksson, Philipp Jell, Mattias Jonsson, Klaus K\"unnemann and Joe Rabinoff for many helpful discussions in relation to this work.  We thank the referee for the careful reading and helpful comments.

\bigskip
\subsection*{Notations and conventions}

The set of natural numbers $\N$ includes $0$.  
The rings in this paper are usually commutative and with $1$. If $A$ is such a ring and $a\in A$, then  $(a)$ denotes the ideal of $A$ generated by $a$. If $P$ is an abelian group, then $P_A \coloneqq P \otimes_\Z A$ denotes the  $A$-module obtained by base change.

As a base field, we consider usually a \emph{non-Archimedean field} $K$. This means in the whole paper that $K$ is endowed with a non-Archimedean {\it complete} absolute value $\abs$ which is \emph{non-trivial}. 
We use $\Ko \coloneqq \{\alpha \in K \mid 1 \geq |\alpha|\}$ for the {\it valuation ring}, $\Koo \coloneqq  \{\alpha \in K \mid 1 > |\alpha|\}$ for the maximal ideal in $\Ko$ and $\Kt \coloneqq \Ko/\Koo$ for the {\it residue field}. If $M$ is a finitely presented torsion module over $\Ko$, then we denote by $\length(M)$ the {\it content} of $M$. This is a generalization of the length in our non-Noetherian situation which we will introduce in \S~\ref{lattices and content}.

Let $X$ be a scheme. 
If $\Fcal$ is a coherent sheaf on a scheme $X$ and $D$ is a Cartier divisor on $X$, 
we write $\Fcal(D)$ for $\Fcal \otimes_{\Ocal_X} \Ocal_X(D)$. 
We will use additive 
notation for line bundles. If  $L,M$ are line bundles on $X$, then $L+M$ denotes the tensor product of the line bundles $L$ and $M$.
For $m \in \Z$, we denote the $m$-th tensor power of $L$ by $mL$.

Let $X$ be an $n$-dimensional proper scheme over the field $F$.  For line bundles $L_1, \ldots, L_n$, we denote the  {\it degree} of $X$ with respect to $L_1,\ldots,L_n$ by $\deg_{L_1,\ldots,L_n}(X)$. It is given by the intersection number
$$\deg_{L_1,\ldots,L_n}(X)=\mathrm{c}_1(L_1)\cdots \mathrm{c}_1(L_n)\cdot [X].$$
If all line bundles agree with $L$, then we use $\deg_L(X) \coloneqq \deg_{L_1,\ldots,L_n}(X)$ as a shorthand.  
We set $$h^q(X,L) \coloneqq  \dim H^q(X,L).$$

If $X$ is a proper scheme over the non-Archimedean field $K$, then we denote by $\Xan$ its Berkovich analytification. A continuous metric $\phi$ on $L$ means  that $\phi$ is continuous with respect to the Berkovich topology. Again, we use additive notation for metrics which means that the tensor metric of metrics $\phi$ and $\psi$ is denoted by $\phi + \psi$ (see~\ref{notation metric} for more). The associated norm on fibers of $L$ is denoted by $\abs_\phi$ and  $\metr_\phi$ denotes the supremum seminorm on the space $H^0(X,L)$ of global sections.  Usually, we consider $\metr_\phi$  if $X$ is reduced, and it is then a norm.

\section{Preliminaries}
\label{section preliminaries}

In this section, we fix an arbitrary non-Archimedean field $K$. We collect here some background results on the content of a module, models, Monge--Amp\`ere measures, energy and non-Archimedean volumes.

\subsection{$\Q$-line bundles and  positivity in the real Picard group} \label{Q-line bundles, R-nef, R-ample}
In this subsection, we consider a proper, finitely presented scheme $Y$ over $\Ko$. We will recall and fix notations about line bundles. The \emph{special fiber} of $Y$ is $Y_s \coloneqq Y \otimes_\Ko \Kt$.

A {\it $\Q$-line bundle} on $Y$ is a pair $(M,m)$ with $M$ a line bundle on $Y$ and $m \in \N_{>0}$. The set of morphisms between  $\Q$-line bundles $(M,m)$ and $(N,n)$ is given as the inductive limit $$\varinjlim_{k \in \N_{>0}} \Hom(knM,kmN).$$  The tensor product of line bundles induces a tensor product of $\Q$-line bundles. As usual, we denote the $\Q$-line bundle $(M,m)$ by $L$ and we will say that $mL$ is (induced by) an honest line bundle $M$. 
Note that $L=(M,m)$ induces a canonical element 
$$L \otimes 1 \coloneqq M \otimes \frac{1}{m} \in \Pic(Y)_\Q \coloneqq \Pic(Y) \otimes_\Z \Q.$$ 
We call $L$ \emph{nef} (resp.~\emph{ample}) if $M$ is a nef (resp.~ample) line bundle on $Y$. 
These notions are well-defined for elements of $\Pic(Y)_\Q \coloneqq \Pic(Y) \otimes_\Z \Q$.

\begin{definition} \label{recall real ample and real nef}
	By multilinearity, we extend the intersection pairing to define an intersection number $M\cdot C \in \R$ for any $M \in \Pic(Y)_\R$ and any closed curve $C$ in $Y_s$. Then $M \in \Pic(Y)_\R\coloneqq \Pic(Y) \otimes_\Z \R$ is called \emph{nef} if $M\cdot C \geq 0$ for any closed curve $C$ in $ Y_s$.	
	
	We call $M \in \Pic(Y)_\R$ \emph{ample} if there are {ample}  $L_1, \ldots, L_r \in \Pic(Y)$ and $\lambda_1, \ldots, \lambda_r >0$ with $M=\sum_{i=1}^r \lambda_i L_i$ for some non-zero $r \in \N$.
\end{definition}

The  definition of ample is consistent for $\Q$-line bundles, by the Nakai--Moishezon criterion applied to $Y_s$. 
 Obviously, ample implies nef.

\begin{prop} \label{sum of R-nef and R-ample}
	Let $A \in \Pic(Y)_\R$ be ample and $M \in \Pic(Y)_\R$ be nef. Then $A+M$ is   ample.
\end{prop}

\begin{proof}
	The important observation is that $L \in \Pic(Y)$ is nef (resp.~ample) if and only if the restriction $L|_{Y_s}$ to the special fiber  is nef (resp.~ample). This is by definition in the nef case, and it follows from~\cite[Corollaire~9.6.5]{GroEGAIV3} in the ample case.

	Assume first that $A \in \Pic(Y)$ is ample and $M \in \Pic(Y)$ is nef. Then it is well-known that  $A|_{Y_s}+M|_{Y_s}$  is ample and hence $A+M$ is ample. This implies the claim for $\Q$-line bundles.
	
	To show the first claim in general, it is enough to show that $A+M$ is ample for $A \in \Pic(Y)$ ample and $M \in \Pic(Y)_\R$ nef. The existence of the ample line bundle $A$ yields that $Y$ is projective. We conclude that  any element in $\Pic(Y)$ is the difference of two  ample classes~\cite[Corollaire 4.5.8]{ega2} and hence we may assume that $M=\sum_{i=1}^r \lambda_i H_i$ for $H_i \in \Pic(Y)$ ample and $\lambda_i \in \R$. For $\delta > 0$ sufficiently small, we choose $\rho_i \in \Q$ with $\lambda_i < \rho_i < \lambda_i + \delta$. Since all $H_i$ are ample and $M$ is nef, we deduce that $N=\sum_i \rho_i H_i$ is nef. Since the ample cone on $Y_s$ is open, we can choose $\delta \in \Q_{>0}$ so small that the restriction of $A - \delta \sum_i H_i$ to $Y_s$ is ample. By the observation at the beginning, we get that $A - \delta \sum_i H_i$ is ample and hence $A-\delta \sum_i H_i+N$ is ample by the claim for $\Q$-line bundles which we already have shown. Note that  
	$$A+M=A+\sum_{i=1}^r \lambda_i H_i=A-\delta \sum_{i=1}^r H_i+N + \sum_{i=1}^r(\lambda_i + \delta - \rho_i)H_i$$
	and hence $A+M$ is ample as $\lambda_i + \delta - \rho_i>0$.	
\end{proof}

\begin{rem} \label{R-nef limit of Q-ample}
	Let $Y$ be a finitely presented projective scheme over $\Ko$ and let $M$ be a nef element in $\Pic(Y)_\R$. Then  there is a finite dimensional subspace $W$ of $\Pic(Y)_\Q$ such that $M$ is the limit of a sequence of ample elements in $W$. Indeed, the above proof shows that $M$ is the limit of nef elements  $N_k \in \langle H_1, \dots, H_r \rangle_\Q$ and hence $M$ is the limit of the ample classes $N_k + \frac{1}{k}A$.
\end{rem}

\begin{rem} \label{restriction to special fiber}
	For $M \in \Pic(Y)_\R$, it is clear that $M$ is nef if and only if the restriction of $M$ to the special fiber $Y_s$ is nef. The arguments in the proof of Proposition \ref{sum of R-nef and R-ample} show that $M$ is ample if and only if $M|_{Y_s}$ is ample.	
\end{rem}

\subsection{Lattices and content} \label{lattices and content} We will introduce the content of a finitely presented $\Ko$-module as a generalization of the length in our non-Noetherian situation. At the end, we will extend the content to the virtual quotient of two lattices in the same finite dimensional $K$-vector space.

\begin{art} \label{structure thm of finitely presented modules} By~\cite[Proposition 2.10 (i)]{Sch13}, if $M$ is a finitely presented torsion module over $\Ko$, then there is an integer $m$ and some 
	$a_1,\ldots a_m \in \Koo \setminus \{0\}$
	such that
	$$ M \simeq \Ko/( a_1) \oplus \cdots \oplus \Ko / ( a_m)$$
	and the quantities $m$ and $v(a_i)$ are independent of any choice up to reordering.
\end{art}

\begin{defi} \label{length of torsion module}
	Let $M$ be a  finitely presented torsion $\Ko$-module and use the above decomposition.
	We define as in~\cite[\S 2.4]{BE} the {\it content} of $M$  to be 
	$$\length(M) = \sum_{i=1}^m v(a_i) \in \R_{\geq 0}.$$
	In~\cite{Sch13}, the content $\length(M)$ is called the  length of $M$.
\end{defi}

\begin{rem} \label{usual length}
	Recall from~\cite[Ex 2.19]{BE}) that $\length(M)$ agrees with the usual length in case of a discrete valuation with $v(\pi)=1$ for an uniformizing element $\pi$ of $\Ko$.
\end{rem}

\begin{prop}
	\label{prop exactness of length}
	Let $0 \to M_1 \to \cdots \to    M_m \to 0$ 
	be an exact sequence of finitely presented torsion $\Ko$-modules. Then we have
	$$\sum_i (-1)^i \length(M_i) = 0.$$
\end{prop}

\begin{proof}
	This is  a consequence of~\cite[Proposition 2.10]{Sch13}.
\end{proof}

\begin{art} \label{virtual length}
	A finitely generated $\Ko$-submodule $\mathcal V$ of a finite dimensional $K$-vector space $V$ is called a {\it lattice in $V$} if $\mathcal V$ generates $V$ as a $K$-vector space. 
	If $\mathcal V_1,\mathcal V_2$ are lattices in $V$, then there is a lattice $\mathcal V$ of $V$ contained in $\mathcal V_1 \cap \mathcal V_2$. Note that any finitely generated $\Ko$-submodule $\mathcal V$  of $V$ is contained in a free $\Ko$-submodule of finite rank and hence $\mathcal V$ is finitely presented over $\Ko$ by using that $\Ko$ is a coherent ring~\cite[Proposition 1.6]{ullrich95}.  
	For $i=1,2$, it follows that  $\mathcal V_i/\mathcal V$ is a finitely presented torsion module over $\Ko$ and we define the content of the virtual $\Ko$-module $\mathcal V_1/\mathcal V_2$ as 
	$$\length(\mathcal V_1/\mathcal V_2) \coloneqq \length(\mathcal V_1/\mathcal V)- \length(\mathcal V_2/\mathcal V)$$
	by using the content of finitely presented torsion modules over $\Ko$ from Definition~\ref{length of torsion module}. Additivity of the content shows that this is a well-defined real number which might be negative. 
\end{art}

\subsection{Models and metrics}
\label{subsection model and metrics}
In this subsection, we consider a line bundle $L$  on a proper scheme $X$ over $K$.
We will introduce models of $X$ and $L$ defined over the valuation ring $\Ko$. We will see that a model of $L$ induces a metric on $\Lan$. Such metrics are called model metrics. They were introduced by Zhang~\cite{zhang-95} in Arakelov theory  and they  play a similar role in non-Archimedean geometry as smooth metrics in the Archimedean case. 
\begin{art} \label{notation metric}
	As in~\cite[\S 5]{BE}, we use the logarithmic notation for a metric on a line bundle $L$ on $X$: a {\it metric on $L$} is a function $\phi:\Lan \to \R$ such that $\abs_\phi \coloneqq e^{-\phi}$ induces a norm on the $\Hscr(x)$-vector space $L \otimes_X \Hscr(x)$ for every $x \in \Xan$. 
	Here, $\Hscr(x)$ is the completed residue field of $x$ 
	endowed with its canonical absolute value~\cite[Remark 1.2.2]{berkovich-book}.
	The metric is called {\it continuous} if $\phi$ is continuous with respect to the Berkovich topology.  If $\phi$ is a (continuous) metric on $\Ocal_X$, then we identify $\phi$ with the (continuous) function $-\log |1|_\phi$ on $\Xan$.
\end{art}

\begin{defi} \label{models}
	A {\it model $\Xcal$ of $X$} is a flat, proper scheme $\Xcal$ over $\Ko$ together with an identification of its generic fiber $\Xcal_\eta \coloneqq \Xcal \otimes_\Ko K$ with $X$. There is a canonical {\it reduction map} $\redu_\Xcal:\Xan \to \Xcal_s$ to the special fiber $\Xcal_s$ of $\Xcal$  (see~\cite[Remark 2.3]{gubler-martin} and~\cite[\S 2]{gubler-rabinoff-werner2} for details). On closed points of $X$, the reduction is induced by the valuative criterion of properness, and hence coincides with the usual reduction modulo the maximal ideal of $\Ko$.  
	
	We say that a model $\Xcal$ is \emph{integrally closed in $X$} if for every affine open subset $\Ucal$ of $\Xcal$, the ring $\Ocal(\Ucal)$ is integrally closed in $\Ocal(\Ucal_\eta)$, see~\cite[\S 6.3]{ega2} for the integral closure of a scheme.
	
	A {\it model $(\Xcal,\Lcal)$ of $(X,L)$} is a model $\Xcal$ of $X$ and a line bundle $\Lcal$ on $\Xcal$ together with an identification  $\Lcal|_{\Xcal_\eta} \simeq
	L$ compatible with the identification $\Xcal_\eta \simeq X$. We call $\Lcal$ a {\it model of $L$ determined on $\Xcal$}.
\end{defi}

\begin{rem} \label{dominance}
	We say that the model $\Xcal$  is {\it dominated} by a model $\Xcal'$ of $X$ if the identity on $X$ extends to a (unique) morphism $\Xcal' \to \Xcal$ over $\Ko$. This induces a partial order on the set of isomorphism classes of models of $X$. It is easy to show that the isomorphism classes of models of $X$ form a directed system with respect to this partial order.

	If $X$ is projective, then it follows as in~\cite[Proposition 10.5]{gubler-pisa} that the projective models of $X$ are cofinal among all models of $X$. 
\end{rem}

\begin{rem} \label{lattice induced by model}
	We will frequently use in this paper that if $(\Xcal,\Lcal)$ is a model of $(X,L)$, then $H^0(\Xcal,\Lcal)$ is a lattice in $H^0(X,L)$. Indeed, it follows from the direct image theorem given in~\cite[Theorem 3.5]{ullrich95} that $H^0(\Xcal,\Lcal)$ is a finitely presented $\Ko$-module. Using that $K$ is a flat $\Ko$-module, we deduce that $H^0(\Xcal,\Lcal)$ is a lattice in $H^0(X,L)$.	
\end{rem}

\begin{art} \label{metrics induced by model} 
	Let $(\Xcal,\Lcal)$ be a model of $(X,L)$.  Then there is an associated  metric $\phi_\Lcal$ of $L$ determined as follows:  for $x \in \Xan$, pick an open subset $\Ucal$ of $\Xcal$ which contains $\redu_\Xcal(x)$ and which trivializes $\Lcal$. Then $\Lcal(\Ucal) \cong \Ocal(\Ucal)$ mapping $s$ to $\gamma$ and the metric is given in $x$ by $|s(x)|_{\phi_\Lcal}=|\gamma(x)|$. This does not depend on the choice of the trivialization. It is clear that such metrics are continuous on $\Lan$. We refer to~\cite[5.3]{BE} and~\cite[\S 2]{gubler-martin} for more details.
\end{art}

\begin{art} \label{model metrics}
	A metric $\phi$ on $L$ is called a {\it model metric} if there is a non-zero $k \in \N$ and a model $\Lcal$ of $kL$ such that $k\phi = \phi_\Lcal$. We can say that the model metric $\phi$ is given by the $\Q$-model $\Mcal \coloneqq \frac{1}{k}\Lcal$ of $L$ and we will denote it by $\phi_\Mcal$. 
	A model metric is called {\it semipositive} if $\Mcal$ is nef. Then $L$ is nef by~\cite[4.8]{gubler-martin}. 
	 Note that the model $\Lcal$ and hence $\Mcal$ is not unique, but nefness of $\Mcal$ is independent of the choice. For this and
	more details about model metrics, we refer to~\cite{gubler-martin}. 
	
	If $L=\Ocal_X$, then we identify a model metric $\phi$ with the  {\it model function} $ - \log |1|_\phi$. It follows from~\cite[Theorem 7.12]{gubler-crelle} that model functions form a dense $\Q$-subspace of the space of continuous real functions on $\Xan$ and hence the set of model metrics of $L$ is dense in the space of continuous metrics on $\Lan$ with respect to uniform convergence (see~\cite[Theorem 1.2]{gubler-martin} for a generalization).
\end{art}

\begin{example} \label{vertical Cartier divisors}
	A Cartier divisor $D$ on a model $\Xcal$ of $X$ is called {\it vertical} if $D|_X$ is trivial. Then $\Ocal(D)$ is a model of $\Ocal_X$ and we define the {\it model function associated to $D$} by
	$$\phi_D \coloneqq - \log |1|_{\Ocal(D)}.$$
	
	Conversely, for any model $(\Xcal,\Lcal)$ of $(X,\Ocal_X)$, the trivial section $1$ of $\Ocal_X$ extends uniquely to a meromorphic section $s$ of $\Lcal$ and then $D \coloneqq \div(s)$ is a vertical Cartier divisor on $\Xcal$.
\end{example}

\begin{art} \label{bounded metrics}
	We say that a metric $\phi$ on $L$ is \emph{bounded} if for any open subset $U$ which trivializes $L$ and any trivializing section $s \in H^0(U,L)$, the function $-\log|s|_\phi$ is locally bounded on $\Uan$. Clearly, any continuous metric $\psi$ of $L$ is bounded and hence boundedness of $\phi$ is equivalent to $\phi -\psi$ bounded on $\Xan$, by compactness of the latter. Bounded metrics are stable under tensor product, inverse and pull-back. 
	
	For a bounded metric $\phi$ on $L$, we define the \emph{sup-seminorm}
	$$\|s\|_\phi \coloneqq \sup_{x \in \Xan} |s(x)|_\phi$$
	of $s \in H^0(X,L)$. If $X$ is reduced, it is  a norm on $H^0(X,L)$ (see~\cite[Lemma 4.1]{BE}).
	
	For bounded metrics $\phi_1,\phi_2$ on $L$, we define the \emph{distance} 
	$$d(\phi_1,\phi_2) \coloneqq \sup_{x \in \Xan} |(\phi_1-\phi_2)(x)| < \infty.$$
	Recall that we view $\phi_1-\phi_2$ as the (bounded) function on $X^\an$ given as $-\log|1|_{\phi_1-\phi_2}$. Clearly, $d$ is a metric on the space of bounded metrics of $L$ inducing the topology of uniform convergence.
\end{art}

\begin{art} \label{space of small sections}
	For a bounded metric $\phi$ on $L$, we will use the notation
	$$\widehat{H^0}(X,L, \phi ) \coloneqq \{s \in H^0(X,L) \mid \|s\|_\phi \leq 1 \}.$$	
	If the absolute value on $K$ is discrete and $X$ is reduced, then 	$\widehat{H^0}(X,L, \phi )$ is automatically a lattice  in $H^0(X,L)$, \emph{cf.}~\cite[Lemma 1.29 (ii)]{BE}. In the case of a non-discrete valuation, this is not always true. However, it is  true in case of a model metric $\phi_\Lcal$ associated to a model $(\Xcal,\Lcal)$ of $(X,L)$ with $\Xcal$ integrally closed in $X$. Indeed,  Lemma~\ref{global sections and widehat} below yields
	$$\widehat{H^0}(X,L, \phi_\Lcal )= H^0(\Xcal,\Lcal)$$
	and hence the claim follows from Remark~\ref{lattice induced by model}.
\end{art}

\begin{lemma} \label{global sections and widehat}
	Assume that $\Xcal$ is a model of $X$ which is integrally closed in $X$ and that $\Lcal$ is a model of $L$ determined on $\Xcal$, then
	\label{lemma equality}
	$$H^0(\Xcal, \Lcal) = \widehat{H^0}(X,L, \phi_\Lcal)$$
\end{lemma}

\begin{proof}
	It is clear that $\subset $ holds in the claim. To prove the converse, we pick $s \in  \widehat{H^0}(X,L,\phi_\Lcal)$. Locally on an affine open subset $\Ucal = \Spec(A)$ which trivializes $\Lcal$, it is given by $\gamma \in A$.
	It follows from~\cite[Theorem~2.1]{chen2015} or~\cite[Theorem~4.15]{BE}  that $A$ integrally closed in $A \otimes_\Ko K$ yields that $A$ is the unit ball of $A \otimes_\Ko K$ with respect to the sup-seminorm on ${\redu}^{-1}(\Ucal_s)$. Now by definition of a model metric and as  $s \in  \widehat{H^0}(X,L,\phi_\Lcal)$, we know that $\gamma$ is in this unit ball and hence in $A$. This proves $s \in \widehat{H^0}(\Xcal,\Lcal)$. 
\end{proof}

\begin{lemma}
	\label{lemma effective}
	Assume that the model $\Xcal$ of $X$ is integrally closed in $X$. Let $D$ be a vertical Cartier divisor of $\Xcal$.
	Then $D$ is effective if and only if $\phi_D \geq 0$.
\end{lemma}
\begin{proof}
	We note that $D$ is effective if and only if the canonical meromorphic section $s_D$ of $\Ocal(D)$ is a global section. Since $D$ is a vertical Cartier divisor, the restriction of $s_D$ to $X$ is a (nowhere vanishing) global section of $\Ocal_X$ and hence the claim is a special case of Lemma~\ref{lemma equality}.
\end{proof}

\begin{lemma}
	\label{lemma cofinal}
	Assume that $K$ is algebraically closed and that $X$ is reduced (resp.~reduced and projective). 
	Then   models (resp.~projective models) of $X$ which are integrally closed in $X$ are cofinal among models of $X$. 
\end{lemma}

\begin{proof}
	This follows from Remark~\ref{dominance} and the scheme-theoretic version of the reduced fiber theorem of Bosch--L\"utkebohmert--Raynaud given in~\cite[Theorem 4.20]{BE}.
\end{proof}

\subsection{Continuous semipositive metrics and Monge--Amp\`ere measures}
\label{subsection MA measures}

Let $X$ be a proper scheme   over the non-Archimedean field $K$.
Recall from~\ref{model metrics} the definition of a semipositive model metric on a line bundle $L$ over $X$.
To describe canonical metrics  of arithmetic dynamical systems, Zhang~\cite{zhang-95} introduced the following generalization.

\begin{defi} \label{continuous semipositive metrics}
	A continuous metric $\phi$ on $\Lan$ is called {\it semipositive} if $\phi$ is a uniform limit of semipositive model metrics  of $L$ with respect to the distance of uniform convergence from~\ref{bounded metrics}. 	
\end{defi}

If $L$ has a continuous semipositive metric, then  $L$ is nef by~\ref{model metrics}. The converse is not known. 
Note that continuous semipositive metrics are closed under sum, pull-back with respect to morphisms of proper schemes over $K$, and uniform limits. This is easily seen from the fact that the first two properties also hold for model metrics. Recall that we use additive notation for metrics, so sum means the tensor metric.

The following {\it Monge--Amp\`ere measures} were introduced by Chambert-Loir.

\begin{prop} \label{MA measures} There is a unique way to associate to a tuple of continuous  semipositive metrics $\phi_1, \ldots, \phi_n$ on line bundles $L_1,\ldots,L_n$ on a proper scheme $X$ over a non-Archimedean field  $K$ and to an effective $n$-dimensional cycle $Z$ on $X$ a positive Radon measure $\mu$ on $\Xan$, formally denoted by $dd^c \phi_1 \wedge \cdots \wedge dd^c \phi_n\wedge \delta_Z$, such that the following holds:
	\begin{itemize}
		\item[(a)] the measure $\mu$ is multilinear and symmetric in $\phi_1, \ldots, \phi_n$;
		\item[(b)] the measure $\mu$ is additive in $Z$; 
		\item[(c)] if $f:X' \to X$ is a morphism of proper  schemes and if $Z'$ is an $n$-dimensional cycle on $X'$, then we have the projection formula 	
		$$
		f_*\big(dd^c(f^*\phi_1) \wedge \cdots \wedge dd^c(f^*\phi_n) \wedge \delta_{Z'}\big) 
		= dd^c \phi_1 \wedge \cdots \wedge dd^c \phi_n\wedge \delta_{f_*Z'};
		$$
		\item[(d)] $\mu$ depends continuously on the $\phi_i$ with respect to uniform convergence (and weak convergence of Radon measures); 
		\item[(e)] $\mu(\Xan)= \deg_{L_1, \ldots, L_n}(Z)$;
		\item[(f)] the Radon measure $\mu$ is compatible with base change of non-Archimedean fields; 
		\item[(g)] assume that $K$ is algebraically closed, and that $\phi_1, \dots, \phi_n$ are model metrics determined by line bundles $\Lcal_1, \ldots, \Lcal_n$ on a model $\Xcal$ of $X$ with reduced special fiber $\Xcal_s$. Then 
		$$
		dd^c \phi_1 \wedge \cdots \wedge dd^c \phi_n:=dd^c \phi_1 \wedge \cdots \wedge dd^c \phi_n\wedge\d_X= \sum_Y \deg_{\Lcal_1, \ldots, \Lcal_n}(Y) \; \delta_{\xi_Y},
		$$
		where $Y$ ranges over all irreducible components of  $\Xcal_s$ and where $\delta_{\xi_Y}$ is the Dirac measure at the unique point $\xi_Y\in X^\an$ with reduction equal to the generic point of $Y$.
	\end{itemize}	
	Properties (b), (c), (d), (f) and (g) characterize the  positive Radon measures $\mu$ uniquely.
\end{prop}

Chambert--Loir~\cite{chambert-loir-2006} obtained this in the case of a non-Archimedean field with a countable dense subset. In the case of any  algebraically closed  non-Archimedean field, this follows from~\cite[Proposition 3.8]{gubler-compositio}. The general case is easily deduced by base change. Alternatively, this follows from the local approach in~\cite{ChaDuc}, see~\cite[\S 8.1]{BE} for details. 

\begin{rem} \label{real forms and currents}
	Note that the wedge product and the $dd^c$ in the notation $dd^c \phi_1 \wedge \cdots \wedge dd^c \phi_n\wedge \delta_Z$ are a priori purely formal. It is used to stress the analogy to the complex  Monge--Amp\`ere measures where $dd^c \phi = c_1(L,\phi)$ is the first Chern form. It was shown later in ~\cite[\S 6]{ChaDuc}   and  in~\cite[Theorem 10.5]{gubler-kuennemann}  that the measures could really be understood as a product of $(1,1)$-currents similarly to the complex case.  
\end{rem}

\begin{rem} \label{DSP extension}
	A continuous metric $\phi$ on a line bundle $L$ over $X$ is called a {\it DSP metric} if there are continuous semipositive metrics $\phi_1,\phi_2$ on  line bundles $L_1,L_2$ over $X$ such that $L=L_1-L_2$ and $\phi=\phi_1-\phi_2$. By multilinearity, we can uniquely extend the construction of  $\mu = dd^c \phi_1 \wedge \dots \wedge dd^c \phi_n\wedge \delta_Z$ to DSP metrics $\phi_1,\dots, \phi_n$. The resulting Radon measure is no longer positive, but still satisfies (a)--(g). 
	
	If $X$ is projective, every model metric on a line bundle over $X$ is a DSP metric. Indeed, we have seen in Remark~\ref{dominance} that every model metric is determined on a projective model $\Xcal$, and every line bundle on $\Xcal$ is a difference of two ample line bundles on $\Xcal$~\cite[Corollaire 4.5.8]{ega2}.
\end{rem}

We recall the following result of Yuan and Zhang. 

\begin{prop} \label{negative semidefinite form}
	Let $\phi_2, \dots, \phi_n$ be continuous semipositive metrics of line bundles on the projective scheme $X$ over $K$. Then 
	$$(f,g) \mapsto \int_\Xan f dd^c g 	 \wedge dd^c \phi_2 \wedge \cdots \wedge dd^c\phi_n$$
	defines a negative semidefinite symmetric bilinear form on the $\Q$-vector space of model functions on $X$.	
\end{prop}

\begin{proof}
	By base change, we may assume that $K$ is algebraically closed. 
	 Note that the form is given as a limit of intersection numbers on models and then the claim from the local Hodge index theorem from~\cite[Theorem {2.1}]{yuan-zhang}. We refer to \cite[Propositions 2.20, 2.21]{BFJ2} for the argument.
\end{proof}

It follows from Remark~\ref{real forms and currents} that the construction of Monge--Amp\`ere measures is local in the Berkovich topology. All results in~\cite[\S 5]{BFJ2} still hold. For convenience of the reader, we state and  prove  the {\it comparison principle}. The complex analogue is due to Bedford and Taylor and, in the discretely valued case of residue characteristic $0$, it is given in~\cite[Corollary 5.3]{BFJ2}. 

\begin{prop} \label{locality principle}
	Let $\phi,\psi$ be continuous semipositive metrics on $\Lan$. Then we have
	$$\int_{\{\phi<\psi\}}	(dd^c \phi)^n \ge \int_{\{\phi<\psi\}}	(dd^c \psi)^n$$
	 where $\{\phi<\psi\} \coloneqq \{x \in \Xan \mid \phi(x)< \psi(x)\}$.
\end{prop}

\begin{proof} 
	We will follow closely the arguments from the complex case given in~\cite[Corollary 2.4]{BB10}. Let $\varepsilon >0$, then $\psi_\varepsilon \coloneqq \max(\phi, \psi-\epsilon)$ is a continuous semipositive metric by~\cite[Propositions 3.11, 3.12]{gubler-martin}. By Proposition~\ref{MA measures}(e), we have 
	\begin{equation} \label{total mass identity}
	\int_{\Xan}	(dd^c \phi)^n= \deg_L(X)=\int_{\Xan} (dd^c \psi_\varepsilon)^n .
	\end{equation}
	Note that $\psi_\varepsilon = \phi$ on the open subset $\{\phi > \psi - \varepsilon\}$ of $\Xan$ and that $\psi_\varepsilon = \psi - \varepsilon$ on the open subset $\{\phi < \psi - \varepsilon\}$ of $\Xan$. Since these open subsets are disjoint and since formation of the Monge--Amp\`ere measure is local in the Berkovich topology  (\textit{i.e.} compatible with  restriction to open subsets), we get 
	\begin{equation} \label{application of MA local}
	\int_{\Xan} (dd^c \psi_\varepsilon)^n \geq \int_{\{\phi > \psi - \varepsilon\}} (dd^c \phi)^n  + \int_{\{\phi < \psi - \varepsilon\}} (dd^c \psi)^n.
	\end{equation}
	The right-hand side is bounded below by
	\begin{equation} \label{lower bound of RHS}
	\int_{\Xan} (dd^c \phi)^n -\int_{\{\phi > \psi\}} (dd^c \phi)^n + \int_{\{\phi < \psi - \varepsilon\}} (dd^c \psi)^n.
	\end{equation}
	Combining \eqref{total mass identity}, \eqref{application of MA local} and  \eqref{lower bound of RHS} and using  monotone convergence for $\varepsilon \to 0$, we get the claim. 
\end{proof}

\subsection{Energy}
\label{subsection energy}

We recall here the definition of the energy relative to two semipositive continuous metrics on a line bundle. We will see that all relevant properties of the energy from~\cite[\S 6]{BFJ2} hold over any non-Archimedean field $K$. 

In this subsection, we consider a line bundle $L$ on a proper scheme $X$ over $K$ of dimension $n$.  

\begin{defi}
	\label{recall energy}
	The {\it energy} of two {continuous semipositive} metrics $\phi_1,\phi_2$ of $L$ 
	is  
	$$ 
	\en(L, \phi_1,\phi_2) 
	\coloneqq
	\frac{1}{n+1} 
	\sum_{j=0}^n 
	\int_{X^\an}
	(\phi_1-\phi_2)
	(dd^c \phi_1)^{ j} \wedge (dd^c \phi_2)^{n-j} \in \R.
	$$
\end{defi}

Note that in~\cite[\S 3.8]{BJ18}, the energy was normalized by dividing through $\deg_L(X)$ which makes perfect sense in the case of an ample line bundle. Here, we will be also interested in nef line bundles and so we omit this normalization.

\begin{prop} \label{properties of energy}
	Let $\phi_1, \phi_2, \phi_1', \phi_2', \phi_3$ be continuous semipositive metrics of $L$. Let $d$ be the distance on the space of bounded metrics of $L$ introduced in~\ref{bounded metrics}. Then the following holds.
	\begin{itemize}
		\item[(a)] $\en(L,\phi_2,\phi_1)=-\en(L,\phi_1,\phi_2)$ and $\en(L,\phi_1,\phi_1)=0$. 
		\item[(b)] If $\phi_1 \leq \phi_2$, then $\en(L,\phi_1,\phi_3) \leq \en(L,\phi_2, \phi_3)$. 
		\item[(c)] The cocycle rule $\en(L,\phi_1,\phi_2)+\en(L,\phi_2,\phi_3)+\en(L,\phi_3,\phi_1)= 0$ holds.
		\item[(d)] For $a \in \N$, we have the homogenity 
		\label{equation homogeneity volume}
		$\en(a L , a\phi_1, a\phi_2)
		=
		a^{n+1} \en(L, \phi_1,\phi_2)$. 
		\item[(e)] $\en(L,\phi_1+c,\phi_2)= \en(L,\phi_1,\phi_2)+ c \deg_L(X)$ for  $c \in \R$.
		\item[(f)] $\en(L,\phi_1,\phi_2)$ is concave in $\phi_1$. 
		\item[(g)] The function $t \mapsto \en(L,(1-t)\phi_1+t\phi_2,\phi_3)$ is a polynomial in $t\in [0,1] $ of degree  $\leq n+1$.
		\item[(h)] $\frac{d}{dt}|_{t=0} \en(L,(1-t)\phi_1+t\phi_2,\phi_1)= \int_\Xan (\phi_2-\phi_1) \, (dd^c \phi_1)^n$.
		\item[(i)] $\frac{d^2}{dt^2}|_{t=0} \en(L,(1-t)\phi_1+t\phi_2,\phi_1)= n \int_\Xan (\phi_2-\phi_1) \,  dd^c(\phi_2-\phi_1)\wedge (dd^c \phi_1)^{n-1}$. 
		\item[(j)] $\int_\Xan (\phi_1-\phi_2) \, (dd^c \phi_1)^n \leq E(L,\phi_1,\phi_2) \leq \int_\Xan (\phi_1-\phi_2) \, (dd^c \phi_2)^n$.
		\item[(k)] $|\en(L,\phi_1,\phi_2) - \en(L,\phi_1',\phi_2')| \leq  \left(d(\phi_1,\phi_1')+d(\phi_2,\phi_2')\right) \deg_L(X)$.
		\item[(l)] The energy is compatible with base extensions of non-Archimedean fields.
		\item[(m)] If $f:X' \to X$ is a birational proper morphism, then $\en(f^*L,f^*\phi_1,f^*\phi_2)= \en(L,\phi_1,\phi_2)$.
	\end{itemize}
\end{prop}

\begin{proof}
	All these properties follow rather formally from the definition of the energy and the properties of the Monge--Amp\`ere measures given in the previous subsection. For the arguments, we refer to~\cite[\S 6]{BFJ2},~\cite[\S 3.8]{BJ18} and~\cite[Proposition 9.14]{BE}. Note that for (f) and (j), we need the Hodge index result of Yuan and Zhang recalled in Proposition~\ref{negative semidefinite form}.
\end{proof}

\subsection{Non-Archimedean volumes}
\label{subsection comparison volumes}

In this subsection, we denote by $X$ a reduced proper scheme over $K$.
We consider a line bundle $L$ on $X$ and set $N_m \coloneqq h^0(X,mL)$ for any $m \in \N$.

\begin{art} \label{relative volume}
	Let $V$ be an $N$-dimensional $K$-vector space. Then a norm $\metr$ on $V$ induces a {\it determinant norm} $\det(\metr)$ on the determinant line $\det(V)= \Lambda^N(V)$ by 
	$$\det(\| \tau \|)  \coloneqq \inf_{\tau = v_1 \wedge \dots \wedge v_N} \|v_1\| \cdots \|v_N\|$$
	for any $\tau \in \det(V)$ (see~\cite[\S 2.1]{BE} for details). 	We define the {\it relative volume} of norms $\metr_1, \metr_2$ on $V$ by 
	$$\vol(\metr_1,\metr_2) \coloneqq \log  \left( \frac{\det(\metr_2)}{\det(\metr_1)} \right).$$
	Note here that $\det(\metr_2) / \det(\metr_1)$ is a well-defined positive number since $\det(V)$ is a one-dimensional $K$-vector space. For more details on relative volumes, we refer to~\cite[\S 2.3]{BE}.	
\end{art}

\begin{rem} \label{relative volume and length}
	Recall from~\cite[\S 1.7]{BE} that a lattice $\Vcal$ in $V$ has an {\it associated lattice norm} $\metr_\Vcal$ on $V$ given by 
	$$\| v \|_\Vcal \coloneqq \inf_{\alpha \in K, \,  v \in \alpha \Vcal}  |\alpha|$$
	for any $v \in V$. 
	It follows from ~\ref{structure thm of finitely presented modules} that a finitely presented torsion  $\Ko$-module $M$ is given by $M=\Vcal_1/\Vcal_2$ for two lattices $\Vcal_1 \supset \Vcal_2$ of a finite dimensional $K$-vector space $V$. Conversely, any such quotient is evidently a finitely presented torsion $\Ko$-module. Then~\cite[Lemma 2.20]{BE} yields
	$$\length(M)= \vol(\metr_{\Vcal_1},\metr_{\Vcal_2}).$$
	For any lattices $\Vcal_1,\Vcal_2$ in $V$, additivity of the relative volume~\cite[Proposition 2.14(i)]{BE} yields
	$$\length(\Vcal_1/\Vcal_2)= \vol(\metr_{\Vcal_1},\metr_{\Vcal_2})$$
	where the content of the virtual $\Ko$-module $\Vcal_1/\Vcal_2$ was defined in~\ref{virtual length}.
\end{rem}

\begin{defi} \label{defi volumes}
	Let $\phi, \psi$ be bounded metrics on $\Lan$. Then we define the {\it non-Archimedean volume of $L$ with respect to $\phi$ and $\psi$} by 
	\begin{equation} \label{non-Archimedean volume} 
	\vol(L,\phi, \psi) \coloneqq \limsup_{m \to \infty} \frac{n!}{m^{n+1}} \vol ( \| \cdot \|_{m\phi}, \| \cdot \|_{m\psi} ) = \vol(L) \cdot \limsup_{m\to \infty} \frac{1}{mN_m} \vol ( \|\cdot  \|_{m\phi}, \| \cdot \|_{m\psi} ).
	\end{equation}
	using the relative volume of the sup-norms  $\| \cdot \|_{m\phi}, \| \cdot \|_{m\psi}$ on $H^0(X,mL)$ from~\ref{bounded metrics}. 
	
\end{defi}

It follows from homogeneity and monotonicity of the relative volume of norms in \cite[Proposition 2.14]{BE} that both limsup's are finite. By the right-hand equality, non-Archimedean volumes are thus interesting only when the line bundle $L$ is \emph{big}. 

We will see in Theorem~\ref{existence of the limit} that the limsup in~\eqref{non-Archimedean volume} is a limit if $X$ is geometrically reduced. 

The non-Archimedean volume has the following basic properties holding for any proper reduced $X$. Recall that $d$ denotes the distance on the space of bounded metrics of a given line bundle (see~\ref{bounded metrics}).

\begin{prop} \label{basic properties of non-archimedean volume}
	Let $\phi,\psi,\phi', \psi'$ be bounded  metrics on $\Lan$. Then we have: 
	\begin{itemize}
		\item[(a)] $\vol(L,\phi,\phi)=0$. 
		\item[(b)] The non-Archimedean volume $\vol(L,\phi,\psi)$ is increasing in $\phi$ and decreasing in $\psi$.
		\item[(c)] For any $c \in \R$, we have $\vol(L,\phi+c,\psi)=   \vol(L,\phi,\psi)+ c\vol(L)$.
		\item[(d)] $\vol(L,\phi,\psi) \leq d(\phi,\psi) \vol(L)$.
		\item[(e)] $|\vol(L,\phi,\psi)-\vol(L,\phi',\psi')| \leq \left(d(\phi,\phi')+d(\psi,\psi')\right) \vol(L)$.	
	\end{itemize}	
\end{prop}

\begin{proof} 
	Properties (a), (b) are obvious and (c), (e)  follow from~\cite[Proposition 2.14(ii),(iv)]{BE}. Finally (d) is a consequence of (a) and (e).
\end{proof}

\begin{lemma} \label{lemma comparison}
	Let $\phi, \psi$ be bounded metrics on $L^{\an}$. We suppose that either the valuation on $K$ is discrete or that $\phi,\psi$ are model metrics  induced by line bundles on models which are integrally closed in $X$. Then 
	
	\begin{equation}
	\label{eq comparison volumes}
	\length \left( \frac{\widehat{H^0}(X, mL, m\phi)}{\widehat{H^0}(X, m L, m\psi )} \right)
	=
	\vol ( \| \cdot \|_{m\phi}, \| \cdot \|_{m\psi} ) + O(N_m).
	\end{equation}
\end{lemma}

\begin{proof}
	By~\ref{space of small sections}, we have that $B_m \coloneqq \widehat{H^0}(X,L^{\otimes m}, m\phi)$ and $B_m' \coloneqq \widehat{H^0}(X,L^{\otimes m}, m\psi )$ are lattices in $H^0(X,mL)$. Let $\metr_{B_m}$ and $\metr_{B_m'}$ be the associated lattice norms.  Then we have 
	\begin{equation}
	\label{eq volumes}
	\vol(\| \cdot \|_{B_m}, \| \cdot \|_{B_m'}) = \vol(\metr_{m\phi}, \| \cdot \|_{m\psi}) +O(N_m).
	\end{equation}
	When $K$ is densely valued, \eqref{eq volumes} is actually an equality as then obviously $\| \cdot \|_{B_m}=\metr_{m\phi}$ and $\| \cdot \|_{B_m'}=\| \cdot \|_{m\psi}$. When $K$ is discretely valued, \eqref{eq volumes} holds by ~\cite[Proposition 2.21]{BE}. Now the claim follows from \eqref{eq volumes} and Remark~\ref{relative volume and length}.
\end{proof}

\begin{rem} \label{compare volumes}
	If the valuation on $K$ is discrete or if $\phi,\psi$ are model metrics  induced by line bundles on models which are integrally closed in $X$, then  Lemma~\ref{lemma comparison} implies 
	\begin{equation} \label{non-Archimedean volume_BGJKM} 
	\vol(L,\phi, \psi) = \limsup_{m \to \infty} \frac{n!}{m^{n+1}}  \length \left( \frac{\widehat{H^0}(X,mL, m\phi)}{\widehat{H^0}(X,mL, m\psi )} \right).
	\end{equation}
	In the case of a discretely valued field with $v(\pi)=1$ for an uniformizer $\pi$, we conclude that $\vol(L,\phi, \psi)$ agrees with the non-Archimedean volume considered  in~\cite{BGJKM}. 
\end{rem}

By Chow's lemma, the following result can be used to reduce to the case of projective schemes.

\begin{lemma} \label{volumes and birationality}
	Let $f:X' \to X$ be a birational map of proper reduced schemes over $X$. For bounded metrics $\phi,\psi$ of the line bundle $L$ over $X$, we have
	$$\vol(\metr_{mf^*\phi},\metr_{mf^*\psi})=\vol(\metr_{m\phi},\metr_{m\psi})+o(m^{n+1})$$
	for $m \to \infty$ and hence
	$$\vol(f^*L,f^*\phi,f^*\psi)=\vol(L,\phi,\psi).$$	
\end{lemma}

\begin{proof}
	This follows from~\cite[Lemma 9.11(v)]{BE} and its proof. Note that the projectivity assumption there can be replaced by properness.	
\end{proof}

Next, we give the behaviour of the relative volumes in Definition~\ref{defi volumes} with respect to a base extension $F/K$ of non-Archimedean fields. We recall from~\cite[Definition 1.24, Proposition 1.25]{BE} that the base change of an ultrametric norm $\metr$ on a  $K$-vector space $V$ is the ultrametric norm on the $F$-vector space $V_F \coloneqq V \otimes_K F$ given for $w \in  V_F$ by 
$$\|w\|_F \coloneqq \inf_{\sum \alpha_i v_i =w} \max_i |\alpha_i| \|v_i\|$$
where the infimum runs over all finite decompositions $\sum \alpha_i v_i =w$  with $\alpha_i \in F$ and $v_i \in V$. 
We will use similar notation to denote base changes of schemes and line bundles.

\begin{prop}
	\label{cor vol comparison base change}
	We assume that $X$ is geometrically reduced. 
	Let $\phi, \psi$ be continuous metrics on $L^{\an}$ and let 
	$F/K$ be a non-Archimedean field extension.
	Then
	$$
	\vol ( \| \cdot \|_{m\phi}, \| \cdot \|_{m\psi}     ) 
	=
	\vol ( \| \cdot \|_{m\phi_F},  \| \cdot \|_{m\psi_F}) + o(mN_m)
	$$
	for $m \to \infty$ and hence $\vol(L_F,\phi_F,\psi_F)=\vol(L,\phi,\psi)$.
\end{prop}

\begin{proof} 
	This follows from~\cite[Lemma 9.4]{BE}. Again, projectivity is not used there. 
\end{proof}

The next result is a consequence of the limit theorem of  Chen and Maclean ~\cite[Theorem 4.3]{chen-maclean} as shown in~\cite[\S 9.2]{BE}.

\begin{theo} \label{existence of the limit}
	Let $L$ be a line bundle on a geometrically reduced proper scheme $X$ over $K$ and let $\phi,\psi$ be bounded metrics on $\Lan$. Then the limsup in the definition of the non-Archimedean volume $\vol(L,\phi,\psi)$ is a limit, \textit{i.e.}
	$$\vol(L,\phi, \psi) = \lim_m \frac{n!}{m^{n+1}} \vol ( \| \cdot \|_{m\phi}, \| \cdot \|_{m\psi} ).$$
	The same holds in \eqref{non-Archimedean volume_BGJKM}.
\end{theo}

\begin{proof}
	We reduce to $X$ projective by Lemma~\ref{volumes and birationality}.    Then the first claim follows from~\cite[Theorem 9.8]{BE}. The second claim follows from the first claim and \eqref{eq comparison volumes}.
\end{proof}

The existence of the limit has the following obvious consequences. They simplify the proofs of the main results in this paper quite a lot, however one could also prove them without Theorem~\ref{existence of the limit} and without Corollary~\ref{lemma homogeneity} similarly as  in~\cite{BGJKM}.

\begin{cor} 
	\label{lemma homogeneity}
	Let $L$ be a line bundle on an $n$-dimensional geometrically reduced proper scheme $X$ over $K$ and let $\phi,\psi,\phi_1,\phi_2,\phi_3$ be bounded metrics on $L^{\an}$. Then we have:
	\begin{itemize}
		\item[(i)] $\vol(L,\phi,\p)=-\vol(L,\p,\phi)$
		\item[(ii)] $
		\vol(L,\phi_1,\phi_2)+\vol(L,\phi_2,\phi_3)+\vol(L,\phi_3,\phi_1)=0$.
		\item[(iii)]  $\vol(aL,a\phi,a\p)=a^{n+1}\vol(L,\phi,\p)$ for all $a \in \N$. 
	\end{itemize}
\end{cor}

\section{Volumes on torsion schemes}
\label{section volume torsion}

Let $K$ be any non-Archimedean field, with valuation $v$. In this section, we consider schemes over $\Ko$ which have non-trivial $\Ko$-torsion. We call them torsion schemes over $\Ko$. Our main examples are closed subschemes  of a scheme over $\Ko$ with support in the special fiber. We will see that the space of global sections of a line bundle $L$  over a finitely presented  projective torsion scheme over $\Ko$ is a finitely presented torsion $\Ko$-module and hence we can define the volume of $L$ by mimicking the classical construction from algebraic geometry by using the content from \S \ref{lattices and content} instead of the dimension. We will prove some basic properties of the volume analogously to \cite[Section 3]{BGJKM}. The main difficulty here is that our torsion schemes are not Noetherian unless  the valuation is discrete.

\subsection{Hilbert--Samuel theory} \label{Hilbert-Samuel theory}
We introduce torsion schemes over $\Ko$, define the numbers $h^q(Y,\Fcal)$ for a coherent sheaf $\Fcal$ over such a torsion scheme $Y$, and show that they fit in the usual Hilbert--Samuel theory.

\begin{defi} \label{definition torsion scheme}
	We say that a  $\Ko$-scheme $Y$ is a {\it torsion  scheme} over $\Ko$ if $A$ is a torsion $\Ko$-module
	for any open affine subset $\Spec( A)$ of $Y$. 
\end{defi}

\begin{rem} \label{torsion for finite type}
	Now assume that $Y$ is a torsion scheme of finite type over $\Ko$. Then for any open affine subset $\Spec( A)$ of $Y$ there is a non-zero element $a \in \Ko$ such that 
	$a \cdot A = 0$. Moreover, $Y$ is a torsion scheme over $\Ko$ if and only if 
	there exists some non zero $b \in \Ko$, some  scheme $Y'$ over $\Spec \Ko/( b)$ such that $Y'$ is isomorphic to the $Y$ as a scheme over $\Spec \Ko$.
	If $Y$ is projective {(resp.~proper)}, then $Y'$ is projective {(resp.~proper)}  as well.
\end{rem}

\begin{art}
	Let $\alpha \in \Ko$ and let $A \coloneqq \Ko / (\alpha)$. 
	We pick an $n \in \N$ and set $S \coloneqq A[x_0, \dots, x_n]$. We consider the standard $\N$-grading on $S$.
\end{art}

\begin{defi}
	Let $n\in \N$ and $M$ be an $S$-graded module of finite presentation.
	We denote by $M_j$ the elements of $M$ of degree $j \in \N$.  
	Note that $M_j$ is a finitely presented torsion $\Ko$-module and hence we may use the content from   \S \ref{lattices and content} to define 
	the {\it Hilbert function of $M$} by
	$$ P_M(j) \coloneqq \length(M_j).$$
\end{defi}

\begin{lemma}
	\label{lemma hilbert theory ses}
	{Let $0 \to M' \to M \to    M'' \to 0$ 
		be a short exact sequence of finitely presented graded $S$-modules. 
		Then 
		we have
		$P_{M'}+ P_{M''} = P_{M}$.}
\end{lemma}

\begin{proof}
	This follows easily from Proposition \ref{prop exactness of length}.
\end{proof}

\begin{lemma}
	Let $\psi : M \to N$ be a morphism of finitely presented $S$-modules. 
	Then $\Ker(\psi)$ and $\Coker(\psi)$ are finitely presented $S$-modules.
\end{lemma}
\begin{proof}
	It follows from \cite[Example 3.3]{ullrich95} that the ring $A$ is stably (universally) coherent and hence the lemma follows from standard properties of modules over coherent rings  
	(see for instance \cite[Tags 05CX and 05CW]{stacks-project}).
\end{proof}

\begin{rem} \label{stably coherent}
	Ullrich's results  \cite[Example 3.3]{ullrich95} show that every finitely presented $\Ko$-algebra is coherent, and hence that the structure sheaf of any finitely presented scheme $Y$ over $\Ko$ is coherent. As a result, an $\Ocal_Y$-module is coherent if and only if it is finitely presented. This yields that the pull-back of a coherent module with respect to a morphism of finitely presented (torsion) schemes over $\Ko$ is again coherent (see \cite[\S 0.5.3]{egaI}). The  direct image theorem \cite[Theorem 3.5]{ullrich95} with respect to a proper morphism of finitely presented schemes over $\Ko$ holds. In particular, for a coherent sheaf $\Fcal$ on a finitely presented proper scheme $Y$ over $\Ko$, all cohomology groups $H^i(Y,\Fcal)$ are finitely presented $\Ko$-torsion modules. 
\end{rem}

\begin{defi} \label{definition Euler characteristic}
	Let $\Fcal$ be a coherent sheaf on a finitely presented  {proper} 
	torsion scheme over $\Ko$. 
	Then we define  
	$$h^q(Y,\Fcal) \coloneqq \length(H^q(Y,\Fcal))$$
	for any $q \in \N$ and the {\it Euler characteristic} 
	$$\chi(Y,\Fcal) \coloneqq \sum_{q=0}^{\dim(Y)} (-1)^q	h^q(Y,\Fcal).$$
\end{defi}

\begin{prop} \label{multivariable Euler charactersitics}
	Let $Y$ be a finitely presented projective torsion scheme over $\Ko$ and let $L_1, \ldots, L_r$ be line bundles on $Y$. Then for any coherent sheaf $\Fcal$ on $Y$, the Euler characteristic $\chi(Y,\Fcal(m_1L_1+ \dots + m_rL_r))$ is a polynomial function of $(m_1, \dots, m_r) \in \Z^r$.
\end{prop}

\begin{proof}
	In a first step, we assume that $L_1, \dots, L_r$ are very ample, and follow the lines of the classical proof.
	We have a closed embedding $i:Y \to \Pbb_{\Ko}^{n_1} \times \dots \times \Pbb_{\Ko}^{n_r}$ with  $L_i \cong i^*\Ocal(e_i)$   and using that cohomology does not change after passing to the push-forward $i_*$, we may assume that $Y = \Pbb_{\Ko}^{n_1} \times \dots \times \Pbb_{\Ko}^{n_r}$ and $L_i=\Ocal(e_i)$ for $i=1, \dots, r$ where $e_1,\dots,e_r$ is the standard basis of $\Z^r$. The homogeneous coordinates on $\Pbb_{\Ko}^{n_i}$ are denoted by $y_{i0},\dots, y_{in_i}$. We proceed by induction on $N \coloneqq \sum_{i=1}^r n_{i}$.  
	
	If $N=0$, then $Y = \Pbb_{\Ko}^0$  and all line bundles $L_i$ are trivial, hence  $\chi(Y,\Fcal(m_1L_1+ \dots + m_rL_r))$ is a constant function.
	So we may assume $N>0$ and that the claim holds for $N-1$. We may assume that $n_i>0$ for all $i=0, \dots, r$. Then we consider the morphism
	$\psi:\Fcal(-e_i) \to \Fcal$ 
	induced by multiplication with $y_{in_i}$ and  
	the induced short exact sequence
	$$0 \to \Ker(\psi) \to  \Fcal(-e_i) \to \Fcal \to \Coker(\psi) \to 0.$$
	It is clear that $\Ker(\psi)$ and $\Coker(\psi)$ are coherent sheaves on $Y=\Pbb_{\Ko}^{n_1} \times \dots \times \Pbb_{\Ko}^{n_r}$. In fact, the definition of $\psi$ as multiplication by $y_{in_i}$ yields easily that they are defined on the closed subscheme $y_{in_i}=0$ which is isomorphic to $\Pbb_{\Ko}^{n_1} \times \dots \times \Pbb_{\Ko}^{n_i-1}  \times \dots \times \Pbb_{\Ko}^{n_r}$. 
	By induction hypothesis, we now get  the claim for the coherent sheaves $\Ker(\psi)$ and $\Coker(\psi)$. We twist the above exact sequence with $m_1L_1+ \dots + m_rL_r$ and then we apply additivity of the Euler characteristic to the resulting exact sequence to deduce that 
	$$\chi(Y,\Fcal(m_1L_1+ \dots + m_rL_r))-\chi(Y,\Fcal(m_1L_1+ \dots + m_rL_r - L_i))$$
	is equal to $$\chi(Y,\Coker(\psi)(m_1L_1+ \dots + m_rL_r))-\chi(Y,\Ker(\psi)(m_1L_1+ \dots + m_rL_r))$$
	and hence it is a polynomial function. 
	Now we use the following fact for any function $f:\Z^r \to \R$. If we know that $f(x)-f(x-e_i)$ is a polynomial function for all $i=1, \dots, r$, then it is quite easy to see that $f(x)$ is a polynomial function. 
	This implies that $\chi(Y,\Fcal(m_1L_1+ \dots + m_rL_r))$ is a polynomial function in $(m_1, \dots, m_r)$, proving the first step.
	
	We now consider the general case. By~\cite[Theorem 13.59]{goertz-wedhorn-1}, we can find very ample line bundles $A_i,B_i$ such that $L_i=A_i - B_i$. The first step shows that
	$$\chi(Y, \Fcal(p_1A_1+q_1B_1+ \dots + p_rA_r+q_rB_r))$$
	is a polynomial function in $(p_1,q_1, \dots, p_r,q_r) \in \Z^{2r}$. Applying this with $p_i=m_i$ and $q_i=-m_i$ for all $i=1, \dots, r$, we get the claim. 
\end{proof}

\begin{lemma}
	\label{lemma uniform Serre vanishing}
	Let $Y$ be a finitely presented projective torsion scheme over $\Ko$ and let $L_1, \dots, L_r$ be ample line bundles on $Y$. 
	Then for any coherent sheaf $\Fcal$ on $Y$, we have 
	$$H^q(Y,\Fcal(m_1L_1+ \dots + m_rL_r)) = 0$$
	for all $m_1,\ldots, m_r \in \N$ with  $m_1+ \dots+ m_r$ sufficiently large and all $q>0$.
\end{lemma}

\begin{proof}
	 For $r=1$, this is Serre's vanishing theorem, \emph{cf.}~\cite[Proposition 3.6]{ullrich95}. 
	We  prove now the claim for $r \geq 2$. We note that there is $k \in \N\setminus \{0\}$ such that $H_j \coloneqq kL_j$ is very ample for every $j=1,\dots,r$. Writing $m_j= m_j'k + d_j$ for $m_j' \in \N$ and $d_j \in \{0,\dots,k-1\}$, we see that it is enough to prove the claim for the finitely many coherent sheaves $\Fcal(d_1L_1+\dots+d_rL_r)$ and the very ample line bundles $H_j$. We conclude that we may assume all $L_j$ very ample.
	
	 There is 
 a suitable closed immersion into a multiprojective space $\mathbb P = \mathbb P_{\Ko}^{n_1} \times \dots \times \mathbb P_{\Ko}^{n_r}$ over $\Ko$ such that  $L_j=\Ocal_{\mathbb P}(e_j)|_Y$ for $j=1,\dots, r$, where $\Ocal_{\mathbb P}(e_j)$ is the pull-back of $\Ocal_{\mathbb P_{\Ko}^{n_j}}(1)$ with respect to the $j$-th projection. We choose such an embedding with $\mathbb P$ of minimal dimension.  
 The proof will run by induction over $\dim(\mathbb P)$. If $\dim(\mathbb P)=0$, then $\dim(Y)=0$ and the claim is obvious.
 
 Now we assume $\dim(\mathbb P) \geq 1$. For every $i \in \{1,\dots,r\}$, we pick a homogeneous coordinate $x_i$ on the factor $\mathbb P_{\Ko}^{n_i}$ and set $\mathbb P_i \coloneqq \{x \in \mathbb P \mid x_i=0\}$ for the associated multiprojective space  with $\dim(\mathbb P')=\dim(\mathbb P)-1$. Then $Y_i\coloneqq Y \times_{\mathbb P} \mathbb P_i$ is a finitely presented torsion scheme over $\Ko$ which is a closed subscheme of $Y$ and of $\mathbb P_i$.  The restriction $\Fcal_i$ of $\Fcal$ to $Y'$ is a coherent sheaf on $Y'$ and the restriction  of $L_j$ to $Y_i$ is a very ample line bundle $L_{ij}$ for $j=1,\dots,r$. By induction and using that the $\ell^1$-norm of $(m_i)$ and its $\ell^\infty$-norm are equivalent,  
 there is $\ell \in \N$ such that 
 \begin{equation} \label{inductive vanishing} H^q(Y_i, \Fcal_i(m_1L_{i1}+ \dots + m_rL_{ir})) = 0
 \end{equation}
 for all $m_1, \dots, m_r \in \N$ with $\max\{m_1,\dots,m_r\} \geq \ell$ and all $i \in \{1,\dots,r\}$. Using the case $r=1$, we may choose $\ell$ so large such that $H^q(Y,\Fcal(m_jL_j))=0$ for all $m_j \geq \ell$ and all $j=1,\dots,r$. 
 We claim that for all $m_1, \dots, m_r \in \N$ with $\max\{m_1,\dots,m_r\} \geq \ell$ we have
 \begin{equation} \label{claim to show for multiprojective Serre vanishing}
 H^q(Y, \Fcal(m_1L_{1}+ \dots + m_rL_{r})) = 0
 \end{equation}  
 which implies the lemma. We argue by contradiction. We choose $m_1,\dots, m_r$ minimal with respect to the  product (partial) order on $\N^r$ such that \eqref{claim to show for multiprojective Serre vanishing} is wrong. Let $m_j$ be the maximal $m_1,\dots, m_r$. There is $i \neq j$ such that $m_i>0$, otherwise \eqref{claim to show for multiprojective Serre vanishing} would be true by the choice of $\ell$ using the case $r=1$. Since $x_i$ is a regular global section of $\Ocal_{\mathbb P}(e_i)$ (see the paragraph before Proposition  \ref{BE correction} for the definition of regular), the sequence 
 $$0 \longrightarrow \Ocal_{\mathbb P}(-e_i) \stackrel{\otimes  x_i}{\longrightarrow} \Ocal_{\mathbb P} \longrightarrow \Ocal_{\mathbb P_i} \longrightarrow 0$$
 on $\mathbb P$ is exact. Then the associated 
 long exact cohomology sequence is 
	$$\cdots \longrightarrow \underbrace{H^q(\mathbb P, \Fcal(me-e_i))}_{ H^q(Y, \Fcal(mL-L_i))} \longrightarrow \underbrace{H^q(\mathbb P, \Fcal(me))}_{H^q(Y, \Fcal(mL))} \longrightarrow \underbrace{H^q(\mathbb P_i,\Fcal(me))}_{H^q(Y_i,\Fcal_i(mL_i))}   \longrightarrow \cdots $$	
where $me \coloneqq m_1e_1+ \cdots + m_r e_r$, $mL \coloneqq m_1L_1+ \cdots + m_rL_r$ and $mL_i=m_1L_{i1}+ \cdots + m_rL_{ir}$. By \eqref{inductive vanishing}, we have $H^q(Y_i,\Fcal_i(mL_i))=0$. Using that $\max\{m_1,\dots,m_r\}$ does not change if we replace $m_i$ by $m_i-1$, minimality yields that $H^q(Y, \Fcal(mL-L_i))=0$ and hence $H^q(Y,\Fcal(mL))=0$. This is a contradiction and the lemma follows.
\end{proof}

\begin{cor}
	\label{cor hilbert polynomial}
	Let $L_1, \dots, L_r$ be ample line bundles on a finitely presented projective torsion scheme $Y$ over $\Ko$. Then for any coherent sheaf $\Fcal$ on $Y$, we have that $h^0(Y,\Fcal(m_1L_1+ \dots + m_rL_r))$ is a polynomial function of $(m_1, \dots, m_r) \in \N^r$ for $m_1+ \dots+ m_r$ large enough. 
\end{cor}

\begin{proof}
	This follows from Proposition \ref{multivariable Euler charactersitics} and Lemma \ref{lemma uniform Serre vanishing}.
\end{proof}

\subsection{Global sections in the non-Noetherian case} \label{Finding sections in the non Noetherian case}

The goal of this subsection is to prove that a line bundle on a finitely presented projective torsion $\Ko$-scheme can be written as a difference 
of effective line bundles which have global sections not vanishing on a given finite set of points of $Y$. 
This fact is established in Corollary \ref{corollary line bundle difference effective}.

\begin{art}
	\label{notations lifiting non Noetherian}
	Let $Y$ be a finitely presented  torsion $\Ko$-scheme.
	We denote by 
	$$Y_s \coloneqq Y \times_{\Spec \Ko} \Spec \Kt$$
	the {\it special fiber of $Y$}. We remark that the induced closed immersion $Y_s \to Y$ 
	induces a homeomorphism of the underlying topological spaces
	$|Y_s| \simeq |Y|$.
\end{art}

\begin{lemma}
	\label{lemma lift section 1}
	Let $Y$ be a {torsion scheme of finite type over $\Ko$}.
	Let $\Fcal$ be a {sheaf} on $Y$ and  
	let $s\in H^0(Y_s, \Fcal)$.
	Then there is some $r \in \R$ with ${0}<r< 1$ such that for any $\beta \in \Ko$ with 
	$r \leq |\beta| < 1$ the section $s$ can be lifted to  
	$H^0(Y_{\Ko/(\beta)}, \Fcal)$ via the canonical morphism 
	$ H^0(Y_{\Ko/(\beta)}, \Fcal) \to H^0(Y_s, \Fcal)$.
\end{lemma}

\begin{proof}
	The isomorphism 
	$ \Kt \simeq \varinjlim_{\beta} \Ko / (\beta)$,  
	{where $\beta \in \Koo$,} 
	induces an isomorphism 
	$$\Ocal_{Y_s} \simeq \varinjlim_{\beta \in \Koo} \Ocal_{Y_{\Ko/(\beta)}}$$
	of abelian sheaves on the topological space $|Y|$. 
	Let us denote by $\Fcal_{Y_{\Ko/(\beta)}}$ (resp. $\Fcal_{Y_s}$) the pull-back of $\Fcal$ to $Y_{\Ko/(\beta)}$ (resp. $Y_s$). Then we get similarly
	$$\Fcal_{Y_s} \simeq \varinjlim_{\beta \in \Koo} \Fcal_{Y_{\Ko/(\beta)}}.$$
	By \cite[Chapter~II, Exercise~1.11]{Hart}, we get that 
	$$H^0(Y, \Fcal_{Y_s} ) \simeq \varinjlim_{\beta} H^0(Y_{\Ko/(\beta)}, \Fcal_{Y_{\Ko/(\beta)}})$$
	which proves the above result.
\end{proof}

\begin{lemma}
	\label{lemma lift section 2}
	{Let $Y$ be a projective torsion scheme over $\Ko$ with a coherent $\Ocal_Y$-module $\Fcal$ and an ample line bundle $A$. Then there is some $r\in \R$ with $0<r<1$}
	such that for any $\beta \in \Ko$ with 
	$r \leq |\beta| < 1$ and for any $m \in \N$, the morphism
	$$H^0(Y_{\Ko/(\beta)} , \Fcal(mA) ) \to H^0(Y_s, \Fcal(mA) )$$
	is surjective.
\end{lemma}

\begin{proof}
	We consider the ring 
	$$R \coloneqq \bigoplus_{m \in \N} H^0(Y_s, mA)$$
	and the $R$-module 
	$$M \coloneqq \bigoplus_{m \in \N} H^0(Y_s, \Fcal(mA)).$$
	Since $A$ is ample on the projective scheme $Y_s$ over the residue field $\Kt$, it follows that $R$ is a 
	$\Kt$-algebra of finite type and that 
	$M$ is a finitely generated $R$-module \cite[Example 1.2.22]{Laz1}.

	Then we first pick  a finite set $\{r_i\}_{i\in I}$ of generators of the $\Kt$-algebra $R$.
	For   $\beta \in \Koo$ with $|\beta|<1$ close enough to $1$, Lemma \ref{lemma lift section 1} shows that we can lift all the $r_i$ to 
	$$R_\beta \coloneqq \bigoplus_{m \in \N} H^0(Y_{\Ko/(\beta)}, mA).$$
	Similarly we pick a finite set $\{m_j\}_{j \in J}$ of generators of the finite $R$-module $M$.
	For   $\beta \in \Koo$ with $|\beta|<1$ close enough to $1$, Lemma \ref{lemma lift section 1} again shows that all the $m_j$ lift to 
	$$M_\beta \coloneqq \bigoplus_{m \in \N} H^0(Y_{\Ko/(\beta)}, \Fcal(mA)).$$
	This proves that for $\beta \in \Koo$ with $|\beta|<1$  close enough to $1$ such that the above liftings are possible, the morphism
	$M_\beta \to M$ is surjective. Clearly, this proves the claim.
\end{proof}

\begin{lemma}
	\label{lemma lift section 3}
	Let $Y$ be a finitely presented projective torsion scheme over  $\Ko$.
	Let $\Fcal$ be a  coherent $\Ocal_Y$-module and let $A$ be an ample line bundle on $Y$. 
	Then there exists some $M \in \N$ such that for any integer $m \geq M$, the morphism
	$H^0(Y, \Fcal(mA) ) \to H^0(Y_s, \Fcal(mA) )$  
	is surjective.
\end{lemma}

\begin{proof}
	First choose some non-zero $\alpha \in \Koo$ such that $Y$ is defined over $\Ko/(\alpha)$. 
	By Lemma \ref{lemma lift section 2}, there is some $r \in \R$ with $|\alpha|<r< 1$ such that for any $\beta \in \Ko$ with 
	$r \leq |\beta| < 1$ and for any $m \in \N$, the morphism
	\begin{equation}
	\label{eq surjective section 1}
	H^0(Y_{\Ko/(\beta)} , \Fcal(mA) ) \to H^0(Y_s, \Fcal(mA) )
	\end{equation}
	is surjective.
	If we tensor the surjective homomorphism $\Ocal_Y \to \Ocal_{Y_{\Ko/(\beta)}}$ of coherent $\Ocal_Y$-modules with $\Fcal$, then we get
	an exact sequence 
	$$ 0 \to  \Mcal \to \Fcal \to  \Ocal_{Y_{\Ko/(\beta)}} \otimes \Fcal \to 0$$
	of coherent $\Ocal_Y$-modules.
	Since $A$ is ample, it follows from \cite[Proposition 3.6]{ullrich95} 
	that  
	\begin{equation}
	\label{eq surjective section 2}
	0 \to  H^0(Y, \Mcal(mA) )  
	\to
	H^0(Y, \Fcal(mA) ) 
	\to 
	H^0(Y, \Ocal_{Y_{\Ko/(\beta)}} \otimes \Fcal(mA) )
	\to 0.
	\end{equation}
	is exact for $m \gg 0$. Since 
	$$
	H^0(Y, \Ocal_{Y_{\Ko/(\beta)}} \otimes \Fcal(mA) )
	\simeq 
	H^0(Y_{\Ko/(\beta)} , \Fcal(mA) )
	$$
	we get the result by combining \eqref{eq surjective section 1} and \eqref{eq surjective section 2}. 
\end{proof}

\begin{lemma}
	\label{lemma lift section 4}
	Let $T$ be a finite subset of a finitely presented projective torsion scheme $Y$ over $\Ko$.
	Let $\Fcal$ be a  
	 line bundle on $Y$ and let $A$ be an ample line bundle on $Y$. 
	Then there is $M \geq 0$ such that for any integer $m\geq M$ there 
	exists a section 
	$s \in H^0(Y, \Fcal(mA) )$ 
	{such that $s(t) \in \Fcal_t/{\mathfrak m_t}\Fcal_t$ is non-zero for any $t \in T$ where ${\mathfrak m_t}$ is the maximal ideal in $\Ocal_{Y,t}$.}
\end{lemma}

\begin{proof}
	Using that $Y_s$ is a projective scheme over the residue field $\Kt$,  \cite[Lemma 3.1.1]{BGJKM} yields that there is $M \geq 0$ such that for any integer $m\geq M$ there is 
	$s \in H^0(Y_s, \Fcal(mA) )$ which does not vanish at any point of $T$.
	Then the claim follows from Lemma \ref{lemma lift section 3}.
\end{proof}

\begin{corollary}
	\label{corollary line bundle difference effective}
	Let $Y$ be a finitely presented projective torsion scheme over $\Ko$.
	Let $T\subset Y$ be finite and let $L$ be a line bundle on $Y$.
	Then there are very ample line bundles $A_1, A_2$ on $Y$ with $L \simeq A_1 - A_2$ such that the line bundles $A_i$ have  global sections $s_i$ {with $s_i(t) \neq 0$} for any $t\in T$.
\end{corollary}

\begin{proof}
	We pick an ample line bundle $A$ on $Y$.
	We apply Lemma \ref{lemma lift section 4} first with $\Fcal = L$. 
	For   $m \gg 0$, we get that $mA + L$ has a global section $s_1$ not vanishing at any point of $T$.
	Then we apply  Lemma \ref{lemma lift section 4} with $\Fcal = \Ocal_Y$. For  $m \gg 0$, we get that  
	$mA$  has a global section $s_2$ not vanishing at any point of $T$.
	Hence for  $m \gg 0 $,  
	$A_1 \coloneqq mA +L$ and $A_2 \coloneqq mA$, 
	we get  the result. Note that for $m \gg 0$ these line bundles are very ample by \cite[Theorem 13.59]{goertz-wedhorn-1}.
\end{proof}

Any global section $s$ of a line bundle $L$ over a scheme $X$ defines a closed subscheme $D$ of $X$. We call $s$ \emph{regular} if $D$ is a Cartier divisor of $X$. We will use the following relative version in case of a flat morphism $\pi:X \to Y$ of schemes. Then $s$ is called \emph{relatively regular} if $D$ is a Cartier divisor of $X$ and if $D$ is flat over $Y$.
We recall the following result from \cite[Proposition A.15]{BE}.

\begin{prop}[Boucksom--Eriksson] \label{BE correction}
	Let $\pi:X \to Y$ be a flat (finitely presented) projective morphism of schemes of finite presentation over $\Ko$ and let $L$ be a $\pi$-ample line bundle on $X$. Then $mL$ has a relatively regular section locally over $Y$ for all $m \gg 0$. 	
\end{prop}

\subsection{Asymptotics}\label{sec:asymp} In this subsection, we will study the asymptotics of cohomology groups to introduce the volume of a line bundle  over a  finitely presented  projective torsion scheme $Y$ over $\Ko$. 
We will use the previous subsection to prove that the volume  increases after adding an effective line bundle. We will explain why the non-Noetherian situation makes this surprisingly hard to prove. We start with a crucial continuity result.

\begin{lemma}
	\label{lemma asymptotic easy}
	Let $\Fcal$ be a coherent sheaf  and let $L_0, \dots, L_r$ be line bundles on $Y$. 
	We set $n \coloneqq \dim(\supp(\Fcal))$. Then for all $m_1, \dots, m_r \in \N \setminus \{0\}$ and $m \coloneqq \sum_i m_i$, we have:
	\begin{equation} \label{cohomology estimate}
	h^q(Y,\Fcal(m_1L_1+ \dots + m_rL_r) ) = O(m^n) 
	\end{equation}
	\begin{equation} \label{difference estimate}
	\left| h^q (Y, 
	\Fcal(L_0+m_1L_1+ \dots + m_rL_r)
	) 
	-
	h^q (Y, \Fcal(m_1L_1+ \dots + m_rL_r)) 
	\right| = O( m^{n-1})
	\end{equation}
\end{lemma}

\begin{proof}
	We note that the support of a coherent sheaf is closed. We use the shorthand notation $$\text{$\mb \coloneqq (m_1, \dots, m_r)$,  
		$\Fcal(\mb) \coloneqq \Fcal(m_1L_1+ \dots + m_rL_r)$ and $\Fcal(\mb,L_0) \coloneqq \Fcal(\mb)\otimes L_0$.}$$  
	We prove both claims simultaneously  by induction on $n \coloneqq \dim(\supp(\Fcal))$. 
	If the support is empty, then $\Fcal=0$ and the claims are obvious as the left hand sides are zero. 
	
	Now we suppose that $n \geq 0$ and that the claims are known for all coherent sheaves whose support have dimension $<n$. 
	
	By Corollary \ref{corollary line bundle difference effective}, there are line bundles $A,B$ with global sections $s_A,s_B$ such that 
	$L_0=A-B$ and the supports of $s_A$ and $s_B$ both do not contain 
	any generic point of $\supp(\Fcal)$. Let $E$ (resp. $F$) be the closed subscheme of $X$ defined by $s_A$ (resp. $s_B$).   
	We conclude that the supports of $\Fcal(\mb) |_{E}$ and of $\Fcal(\mb, A)|_E$ 
	have dimension at most $n-1$ and the same also holds for the restrictions  to $F$. 
	We get an exact sequence
	\begin{equation} \label{seq0}
	0 \longrightarrow \Gcal   \longrightarrow  \Fcal \stackrel{\otimes s_A}{\longrightarrow}   
	\Fcal(A) \longrightarrow \Fcal(A)|_{E} \longrightarrow 0
	\end{equation}
	of coherent sheaves on $Y$. We twist by $m_1L_1+ \dots + m_rL_r$ and get the exact sequence 
	\begin{equation} \label{seq1}
	0 \longrightarrow \Gcal(\mb)   \longrightarrow  \Fcal(\mb) \stackrel{\otimes s_A}{\longrightarrow}   
	\Fcal(\mb,A) \longrightarrow \Fcal(\mb,A)|_{E} \longrightarrow 0.
	\end{equation}
	By the choice of the global section $s_A$, the dimension of $\supp(\Gcal)$ is at most $n-1$. 
	By induction on $n$, we have 
	\begin{equation} \label{indhyp}
	h^q(Y,\Gcal(\mb)) = O(m^{n-1})  \quad   \text{and}  \quad     h^q (E,\Fcal(\mb,A)|_E) = O(m^{n-1}).
	\end{equation} 
	We split the exact sequence \eqref{seq0} into two short exact sequences
	\begin{equation} \label{seqsplit}
	0 \longrightarrow \Gcal   \longrightarrow  \Fcal \longrightarrow  \Hcal 
		\longrightarrow 0 \quad \text{and}\quad 
		0 \longrightarrow \Hcal   \longrightarrow     
		\Fcal(A) \longrightarrow \Fcal(A)|_{E} \longrightarrow 0.
	\end{equation}
	We twist again these two short exact sequences by $m_1L_1+ \dots + m_rL_r$ and then we use the associated  long 
	exact cohomology sequences to deduce  
	\begin{equation*} \label{ineq 1}
	-h^{q-1}(E, \Fcal(\mb,A)|_E) \leq h^q(Y , \Fcal(\mb, A) ) -
	h^q( Y, \Hcal(\mb)) 
	\leq h^q(E, \Fcal(\mb, A)|_E) 
	\end{equation*}
	and
	\begin{equation*} \label{ineq 1'}
	-h^{q}(Y, \Gcal(\mb)) \leq h^q(Y , \Hcal(\mb) ) -
	h^q( Y, \Fcal(\mb)) 
	\leq h^{q+1}(Y, \Gcal(\mb)). 
	\end{equation*}
	Using {these inequalities} and \eqref{indhyp}, we get 
	\begin{equation} \label{ineq 2} 
	h^q(Y , \Fcal(\mb, A) ) -
	h^q( Y, \Fcal(\mb))  =O(m^{n-1}).
	\end{equation}
	We apply \eqref{ineq 2} to $\Fcal' \coloneqq\Fcal(A-B)$ instead of $\Fcal$ 
	and $B$ instead of $A$ to get
	\begin{equation} \label{ineq 3}
	h^q(Y , \Fcal'(\mb, B) ) -
	h^q( Y, \Fcal'(\mb))  = O(m^{n-1}).
	\end{equation}
	Using that $\Fcal'(\mb)  \simeq \Fcal(\mb, L_0)$ and that $\Fcal'(\mb, B) \simeq \Fcal(\mb,A)$, the inequality \eqref{difference estimate} for $n$ follows easily from \eqref{ineq 2} and \eqref{ineq 3}.  
	It is clear that \eqref{cohomology estimate} follows from a repeated application of \eqref{difference estimate} by choosing $L_0$ from $L_1, \dots, L_r$.
\end{proof}

\subsection{Asymptotic cohomological functions on the real Picard group}

This subsection is inspired by the results about $\R$-divisors on reduced projective schemes over a field from~\cite[\S 3]{BGJKM}.

\begin{theo}\label{thm:asympcohom} Let $Y$ be a finitely presented projective torsion scheme over $\Ko$, of dimension $n$. For each $q=0,\dots,n$, there exists a unique function $\hh^q(Y,\cdot):\Pic(Y)_\R\to\R_{\ge 0}$ such that:
	\begin{itemize}
		\item[(i)] for any $L \in \Pic(Y)$, we have 
		$$
		\hh^q(Y,L)=\limsup_{m\to\infty}\frac{n!}{m^n}h^q(Y,mL);
		$$
		\item[(ii)] for all $t\in \R_{\ge 0}$ and $M \in \Pic(Y)_\R$, we have $\hh^q(Y,tM)=t^n\hh^q(Y,M)$; 
		\item[(iii)] the function $\hh^q(Y, \cdot)$ is continuous on any finite dimensional real subspace of $\Pic(Y)_\R$.
	\end{itemize}
	For any presentation of $M\in\Pic(Y)_\R$ as $M=\sum_i x_i L_i$ with $L_i\in\Pic(Y)$ and $x_i\in\R$, we further have
	\begin{equation}\label{equ:hhqlimsup}
	\hh^q(Y,M)=  \limsup_{m \to+\infty} \frac{n!}{m^n} h^q\left(Y, \sum_i \lfloor m x_i \rfloor L_i \right).
	\end{equation}
	If $M$ is nef, then this limsup is a limit, and $\hh^q(Y,M)=0$ for $q>0$. 
\end{theo}
As in the Appendix, we slightly abusively denote by $L\in\Pic(Y)_\R$ the image of $L\in\Pic(Y)$. 

\begin{proof} Set $P:=\Pic(Y)$. By \S~\ref{sec:asymp}, the function $h:P\to\R_{\ge 0}$ defined by $h(L):= n! h^q(Y,L)$ satisfies the assumptions of Theorem~\ref{thm:hh} with $s=n$. The existence and uniqueness of $\hh^q$ are thus direct consequences of Theorem~\ref{thm:hh}, while~\eqref{equ:hhqlimsup} follows from Proposition~\ref{prop:hh}. 
	
	To prove the final point, we may enlarge the set $L_1,\dots,L_r\in\Pic(Y)$ and assume that $L_1$ is ample. Then
	$$
	\sigma:=\left\{x\in\R^r\mid M(x):=\sum_i x_i L_i\text{ ample}\right\}
	$$
	is a non-empty open convex cone, and $M(x)$ is nef if and only if $x\in\overline{\sigma}$ (since $M(x)$ nef implies that $(x_1+\e)L_1+\sum_{i>1} x_i L_i$ is ample for all $\e>0$). For any $x\in\Z^r\cap\sigma$, Lemma~\ref{lemma uniform Serre vanishing} implies $h^q(Y,m M(x))=0$ for $m\gg 1$ if $q>0$, while Corollary~\ref{cor hilbert polynomial}  and Lemma \ref{lemma asymptotic easy} yield
	$$
	\hh^0(Y,M(x))=\lim_{m\to\infty}\frac{n!}{m^n}h^0(Y,m M(x)).
	$$
	The final assertion of the theorem is now a consequence of Proposition~\ref{prop:hh}. 
\end{proof}

\medskip

In the special case $q=0$, we define the \emph{volume} of $M \in \Pic(Y)_\R$ as
$$
\vol(M) \coloneqq \vol(Y,M)  \coloneqq \hat{h}^0(Y,M).
$$

\begin{lemma}
	\label{lemma_inequality_nef_cartier}
	Let $M \in \Pic(Y)_\R$ and 
	let $E$ be a line bundle associated to an effective Cartier divisor. 
	Then 
	\begin{equation} \label{volume increasing}
	\vol(M) \leq \vol(M + E).
	\end{equation}
\end{lemma}

\begin{proof}
	By continuity and homogeneity of the volume on $\Pic(Y)_\R$  shown in Theorem \ref{thm:asympcohom}(ii) and (iii), we may assume that $M=L$ is a line bundle on $Y$.
	By assumption on $E$, there is a regular global section $s\in H^0(Y,E)$. For any $m\in\N$, multiplication by $s^m$  yields an injection of sheaves
	$$ 
	\cO_Y(mL) \hookrightarrow\cO_Y(m(L+E)),
	$$
	inducing an injection of $\Ko$-modules
	$$
	H^0(Y, mL) \hookrightarrow H^0\left(Y, m(L+E)\right). 
	$$
	Thus
	$$\length\left( H^0(Y, mL )\right) \leq \length \left( H^0\left(Y, m(L+E)\right)\right),
	$$
	which implies the claim.
\end{proof}

\begin{rem} \label{open question with volume for ample}
	We were not able to prove that~\eqref{volume increasing} holds in case of $M \in \Pic(Y)_\R$ and an ample line bundle $E$. It would be quite plausible that this holds and in the Noetherian case it is true. In general, the problem is to construct a regular global section of $mE$ for some non-zero $m \in  \Nbb$. 
	 Then  homogeneity of the volume  (see Theorem \ref{thm:asympcohom}(ii)) would give~\eqref{volume increasing}. 
	In Lemma~\ref{lemma nef increasing}, we will solve the problem in a special case which will be enough for our application.	
\end{rem}

\section{Volume formulas for nef models and semipositive metrics}
\label{section formula volume energy}

Let $X$ be a projective scheme of dimension $n$ over a non-Archimedean field $K$.
We denote the valuation of $K$ by $v$. 
In the first subsection, we will first prove an asymptotic volume formula for an effective vertical Cartier divisor $D$ on a model of $X$ in terms of an integral of the model function $\phi_D$ against a Monge--Amp\`ere measure.  

In \cite{BGJKM} and \cite{BE}, it was shown that the non-Archimedean volume $\vol(L,\phi,\phi')$ agrees with the energy for semipositive metrics $\phi,\phi'$ of the line bundle $L$ under certain assumptions. In fact, the additional assumptions in \cite{BGJKM} were that the underlying scheme $X$ is a normal variety  and that the non-Archimedean base field $K$ is discretely valued while in \cite{BE}, the result holds for any non-Archimedean field, but the line bundle $L$ was assumed to be ample and $X$ was assumed to be smooth.  
The goal of the second subsection is to generalize both results allowing any non-Archimedean field and not requiring $L$ to be ample.
Our proof uses the asymptotic volume formula and our previous results on torsion schemes.

\subsection{An asymptotic volume formula}

Let $K$ be any non-Archimedean field.
In this subsection, we generalize a crucial volume formula on an effective vertical Cartier divisor $D$ of a given projective $\Ko$-model $\Xcal$ from the ample to the nef case. Such a formula was obtained in \cite{BGM} by using a change of metric formula in terms of the Deligne pairing from \cite{BE}. We extend the volume formula here to nef line bundles by using a continuity argument for the volumes on  $D$ introduced in Section \ref{section volume torsion} as the effective vertical Cartier divisor $D$ is obviously a finitely presented projective torsion scheme over $\Ko$.

In this special situation, we can really prove the desired volume inequality  mentioned in Remark \ref{open question with volume for ample}:

\begin{lemma}
	\label{lemma nef increasing}
	Let ${\KX}$ be a flat projective and finitely presented scheme over $\Ko$.
	Let $L, E \in \Pic({\KX})_\R$ with $E$ nef.
	Then for any  effective vertical Cartier divisor $D$  on ${\KX}$, we have  
	$$\vol(D, L) \leq \vol(D, L + E).$$
\end{lemma}
\begin{proof}
	By continuity and homogeneity of the volume (\emph{cf.}~Theorem~\ref{thm:asympcohom}), we may assume that $E\in \Pic(\KX)$ is ample and that $L \in \Pic({\KX})$. From Proposition \ref{BE correction}, we deduce that for $m\gg0$ the line bundle $mE$ has a relatively regular section $s$.
	Hence, replacing $E$ by $mE$ and $L$ by $mL$  and using homogeneity of the volume again, we can assume that $E$ has a relatively regular section $s$. 
	It follows from \cite[Lemma 1.3]{BGM} that $s|_D$ is a regular section of $E|_D$. We recall that a regular section has an associated effective Cartier divisor $\div(s|_D)$ and $E|_D$ is isomorphic to $\Ocal(\div(s|_D))$. Then the claim follows from Lemma \ref{lemma_inequality_nef_cartier}. 
\end{proof}

In the following, we consider a flat projective $\Ko$-model $\Xcal$ of the $n$-dimensional projective scheme $X$ over $K$. Let $D$ be an effective vertical Cartier divisor on $\Xcal$. For any continuous function $f:\Xan \to \R$, the integral 
\begin{equation*} \label{extend MA integral to real Picard group}
\int_\Xan f \, dd^c\phi_{\Lcal_1} \wedge \dots \wedge dd^c\phi_{\Lcal_n}
\end{equation*}
is multilinear in $\Lcal_1, \dots, \Lcal_n \in \Pic(\Xcal)$ and hence  extends canonically to a multilinear function on $\Pic(\Xcal)_\R$. 

\begin{prop} \label{volume formula for real nef}
	Let $\Lcal  \in \Pic(\Xcal)_\R$ be nef and let $\phi_D$ be the model function associated to the effective vertical Cartier divisor $D$. Then we have 
	$$\vol(D,\Lcal)= \int_\Xan \phi_D \, (dd^c \phi_\Lcal)^n.$$	
\end{prop}

\begin{proof}
	If $\Lcal \in \Pic(\Xcal)$ is ample, then this is proven in \cite[Theorem 2.4]{BGM}. By homogeneity, we conclude that the claim holds for ample line bundles in $\Pic(\Xcal)_\Q$. In general, 
	Remark \ref{R-nef limit of Q-ample} yields a finite dimensional subspace $W$ of $\Pic(\Xcal)_\Q$ such that the nef  $\Lcal$ is a limit of ample line bundles in $W$ and hence the claim follows from the previous case by continuity. 
\end{proof}

\begin{prop} \label{asymptotic length formula for nef}
	Let $\Xcal$ be a flat projective $\Ko$-model of the $n$-dimensional projective scheme $X$ over $K$. Let $D$ be an effective vertical Cartier divisor on $\Xcal$, let  $\Lcal_0, \dots, \Lcal_r$ be in $\Pic(\Xcal)$ and assume that $\Lcal_1, \dots, \Lcal_r$ are nef. Then we have 
	$$h^0(D,\Lcal_0+m_1 \Lcal_1 + \dots + m_r \Lcal_r) = 
	\frac{1}{n!} \, \int_{X^{\an}} \phi_D \, \left( dd^c \phi_{m_1 \Lcal_1 + \dots + m_r \Lcal_r} \right)^n  + o(m^{n})$$
	for $m_1, \dots, m_r \in \N$ and $m \coloneqq m_1+ \dots + m_r \to \infty$. 	If all  $\Lcal_1, \dots, \Lcal_r$ are ample, then the above asymptotic formula holds even with $o(m^n)$ replaced by $O(m^{n-1})$. 
\end{prop}

\begin{proof}	
	By Lemma \ref{lemma asymptotic easy}, it is enough to consider the case $\Lcal_0= \Ocal_\Xcal$.
	Then we deal first with the case where $\Lcal_1, \dots, \Lcal_r$ are nef.	We argue by contradiction. If the claim is not true, then  there is $\varepsilon >0$ such that for any $i \in \{1, \dots, r\}$ there is  a sequence $(m_i^{(k)})_{k \in \N}$ in $\N$ with $m^{(k)}= m_1^{(k)}+ \dots + m_r^{(k)} \to \infty$ and 
	\begin{equation} \label{from contradiction assumption}
	\left| h^0(D,m_1^{(k)} \Lcal_1 + \dots + m_r^{(k)} \Lcal_r) -
	\frac{1}{n!} \, \int_{X^{\an}} \phi_D \, \left( dd^c \phi_{m_1^{(k)} \Lcal_1 + \dots + m_r^{(k)} \Lcal_r} \right)^n  \right| \geq \varepsilon \cdot(m^{(k)})^n
	\end{equation}
	for all $k \in \N$. Passing to a subsequence, we may assume that $m_i^{(k)}/m^{(k)}$ converges to  a non-negative $x_i \in \R$. 
	We deduce from Theorem~\ref{thm:asympcohom} that 
	\begin{equation} \label{nef limit formula again}
	\vol(D, x_1 \Lcal_1 + \dots + x_r\Lcal_r)	=  \lim_{m \to \infty} \frac{h^0\left(D,  \sum_{i=1}^r \lfloor m x_i \rfloor \Lcal_i \right)}{m^n/n!}.
	\end{equation}
	We pick any $\delta>0$.  Then  we have $|m_i^{(k)}-x_i m^{(k)}| \leq \delta m^{(k)}$ for $k \gg 0$
	and Lemma 
	\ref{lem:perturbation} yields 
	$$h^0\left(D,m_1^{(k)} \Lcal_1 + \dots + m_r^{(k)} \Lcal_r\right) - h^0\left(D, \lfloor m^{(k)}x_1 \rfloor \Lcal_1 + \dots + \lfloor m^{(k)}x_r \rfloor \Lcal_r\right) = \delta \cdot O\left((m^{(k)})^{n}\right)$$
	and hence \eqref{nef limit formula again} leads to 
	$$\left(m^{(k)}\right)^{-n}  h^0\left(D,m_1^{(k)}\Lcal_1 + \dots + m_r^{(k)} \Lcal_r\right) - \frac{1}{n!} \vol(D, x_1 \Lcal_1 + \dots + x_r\Lcal_r)= \Ocal(\delta)$$
	for $k$ sufficiently large. By Proposition \ref{volume formula for real nef}, we may replace $\vol(D, x_1 \Lcal_1 + \dots + x_r\Lcal_r)$ by 
	$$ \int_{X^{\an}} \phi_D \, \left( dd^c \phi_{x_1 \Lcal_1 + \dots + x_r \Lcal_r} \right)^n.$$
	If we choose $\delta$ sufficiently small, we get a contradiction to \eqref{from contradiction assumption}.
	
	If $\Lcal_1, \dots, \Lcal_r$ are ample, then Corollary \ref{cor hilbert polynomial} and Lemma \ref{lemma asymptotic easy} show  that $h^0\left(D,\Lcal_0+m_1 \Lcal_1 + \dots + m_r \Lcal_r\right)$ is a polynomial function of degree at most $n$  in $m_1,\dots,m_r$ for $m \gg 0$. Since  
	$$\int_{X^{\an}} \phi_D \, \left( dd^c \phi_{m_1^{(k)} \Lcal_1 + \dots + m_r^{(k)} \Lcal_r} \right)^n$$
	is also a polynomial function of degree at most $n$, the difference of the two functions is not only of order $o(m^n)$, but even of order $O(m^{n-1})$. 
\end{proof}

\subsection{Comparison between energy and non-Archimedean volume}

The following easy filtration argument will be applied several times.

\begin{lemma} \label{filtration lemma}
	Let $\KX$ be a flat proper scheme over $\Ko$ with a line bundle $\KM$ and an effective vertical Cartier divisor $E$. Then we have
	$$\length \left( \frac{
		H^0(\KX,  \KM + m \Ocal(E) )
	}{
		H^0(\KX,\KM ) 
	} \right) 
	\leq \sum_{i=1}^m \length \left( H^0(E,\KM + i \Ocal(E)) \right)$$
	for any $m \in \N$.
\end{lemma}
Let $X$ be the generic fiber of $\KX$ and let $M= \Mcal|_X$. Recall from Remark \ref{lattice induced by model} that on the left hand side we have a quotient of two lattices in $H^0(X,M)$ with content $\ell$  defined by \ref{virtual length}.

\begin{proof}
	We may assume $m>0$. 
	Then for every $i \in \{1, \ldots m\}$, multiplication with the canonical global section $s_E$ leads to a short exact sequence 
	$$ 0 \longrightarrow \Ocal_{\KX} \stackrel{\cdot s_E}{\longrightarrow} \Ocal(E)    \longrightarrow \Ocal(E)|_E\to 0$$
	inducing a short exact sequence
	$$ 0 \longrightarrow \KM + (i-1) \Ocal(E) \longrightarrow \KM + i\Ocal(E) \longrightarrow (\KM + i\Ocal(E))|_{E}\to 0.$$
	The start of the corresponding long exact cohomology sequence is
	$$0 \longrightarrow H^0({\KX}, \KM + (i-1)\Ocal(E)) \longrightarrow H^0({\KX},  \KM + i\Ocal(E) )\longrightarrow H^0( E  , \KM +i\Ocal(E)) \longrightarrow \ldots $$
	and hence we get
	$$\length \left( \frac{
		H^0({\KX},  \KM +i\Ocal(E) )
	}{
		H^0({\KX}, \KM + (i-1)\Ocal(E)) 
	} \right)
	\leq 
	\length(  H^0( E  , \KM +i\Ocal(E)) ),
	$$
	leading to
	$$\length \left( \frac{
		H^0({\KX},  \KM + m\Ocal(E))
	}{
		H^0({\KX}, \KM ) 
	} \right)
	\leq 
	\sum_{i=1}^{m}\length(  H^0( E  , \KM +i \Ocal(E)) )$$
	and proving the claim.	
\end{proof}

In the following result, we will use the content $\length(\Vcal_1/\Vcal_2)$ of the virtual quotient of two lattices $\Vcal_1,\Vcal_2$ in the same $K$-vector space. We refer to \ref{virtual length} for the definition.

We first deal with the model case. 

\begin{prop}
	\label{prop energy 2'}
	Let $L$ be {a} line bundle on the projective scheme $X$ and let $\KX$ be 
	a projective model of $X$.  
	We consider  
	{nef} models $\KL_1$ and $\KL_2$ of $L$ on $\Xcal$ and 
	we write $\KL_1-\KL_2 = \Ocal(D)$ for 
	some  vertical 
	Cartier divisor $D$ on $\KX$.
	In addition, let $\Mcal$ be a line bundle on $\KX$ with generic fibre $M \coloneqq \Mcal_{|X}$. Then we have 
	\begin{align*}
	E(L, \phi_{\KL_1}, \phi_{\KL_2}) 
	=  
	\lim_{m\to 0} \frac{n!}{m^{n+1}}  
	\length 
	\left( \frac {H^0(\Xcal,\Mcal + m{\KL_1})} {H^0(\Xcal,\Mcal + m {\KL_2})} \right).
	\end{align*}
\end{prop}

\begin{proof}
	We first study what happens if we replace $D$ by $D' \coloneqq D+\div(\pi)$ for any non-zero $\pi \in K$. Then we replace the model $\KL_1$ by the model $\KL_1' \coloneqq \KL_1(\div(\pi)) \simeq \KL_1$ of $L$ which is also nef. By construction, we have $\Ocal(D')=\KL_1'-\KL_2$. Note  $\phi_{\KL_1'}=  v(\pi) + \phi_{\KL_1}$.
	Using Proposition \ref{properties of energy}(e), we get 
	\[
	E(L, \phi_{\KL'_1}, \phi_{\KL_2})  = 
	v(\pi) \deg_L(X) + E(L, \phi_{\KL_1}, \phi_{\KL_2}).
	\] 
	We have
	$$
	\length 
	\left( \frac {H^0(\Xcal,\Mcal + m{\KL'_1})} {H^0(\Xcal,\Mcal + m {\KL_1})}
	\right)
	=
	v(\pi) m\, h^0(X,M + mL) 
	$$
	and hence
	$$
	\length \left( \frac {H^0(\Xcal,\Mcal + m{\KL'_1})} {H^0(\Xcal,\Mcal + m {\KL_2})}
	\right)	
	=
	\length 
	\left( \frac {H^0(\Xcal,\Mcal + m{\KL_1})} {H^0(\Xcal,\Mcal + m {\KL_2})} \right)
	+ C_n(m)
	$$
	with $C_n(m)$ defined as
	$$v(\pi) m\, h^0(X,M \otimes L^{\otimes m}) = 	v(\pi) m\, h^0(X, L^{\otimes m}) + o(m^{n+1})
	=  v(\pi) \vol(L) \frac{m^{n+1}}{n!}    + o(m^{n+1})      $$
	where we used the analogue of \eqref{difference estimate} in Lemma \ref{lemma asymptotic easy} for projective schemes over a field (see \cite[Proposition 3.1.2]{BGJKM}).
	Since $L$ is a nef line bundle on $X$, we know that $\vol(L)=\deg_L(X)$ and hence  the claim for $D'$ implies the claim for $D$. This proves that we can replace $D$ by $D'$.
	
	First, we will prove the inequality 
	\begin{equation} \label{limsup inequality}
	\limsup_{m\to 0} \frac{n!}{m^{n+1}}  
	\length 
	\left( \frac {H^0(\Xcal,\Mcal + m{\KL_1})} {H^0(\Xcal,\Mcal + m {\KL_2})} \right)
	\leq 
	E(L, \phi_{\Lcal_1},\phi_{\Lcal_2}).
	\end{equation} 
	There is a 
	non-zero $\pi$ in  $\Ko$ such that $D+\div(\pi)$ 
	is an effective Cartier divisor. Replacing $D$ by $D + \div(\pi)$, the above shows that it is enough 
	to prove \eqref{limsup inequality} if $D$ is effective. 
	
	For non-zero $m$, Lemma \ref{filtration lemma} applied with $\Mcal + m\KL_2$ instead of $\KM$ and with $D$ instead of $E$ shows  that
	$$\length \left( \frac{
		H^0(\Xcal,  \Mcal + m\KL_1 )
	}{
		H^0(\Xcal, \Mcal + m\KL_2) 
	} \right)
	\leq 
	\sum_{i=1}^{m}\length(  H^0(D  , \Mcal + m \Lcal_2+ i\Ocal(D)) ).$$
	Using $\KL_1-\KL_2 = \Ocal(D)$, we get
	$$\length \left( \frac{
		H^0(\Xcal,  \Mcal + m\KL_1 )
	}{
		H^0(\Xcal, \Mcal + m\KL_2) 
	} \right)
	\leq 
	\sum_{i=1}^{m}\length(  H^0( D  ,\Mcal +  i\KL_1 +(m-i)\KL_2) ).$$
	Since $\Lcal_1$ and $\Lcal_2$ are nef, Proposition \ref{asymptotic length formula for nef}
	shows that the right hand side is
	$$
	\sum_{i=1}^{m}\sum_{j_1+ j_2=n} \frac{i^{j_1}  (m-i)^{j_2}}{j_1!  j_2!}  \int_{X^{\an}} \phi_{D} \, (dd^c \phi_{\Lcal_1})^{j_1} \wedge  (dd^c \phi_{\Lcal_2})^{j_2}
	+ o(m^{n+1}).$$
	We note that the following limit for $m\to\infty$ exists and is given by the 
	sum of Riemann integrals
	$$
	\lim_{m \to \infty} m^{-(n+1)}\sum_{i=1}^{m} i^{j_1}  (m-i)^{j_2}= \lim_{m \to \infty}   \frac{1}{m}\sum_{i=1}^{m}\left(\frac{i}{m}\right)^{j_1}  \left(1-\frac{i}{m}\right)^{j_2}
	=  \int_0^1  t^{j_1} (1-t)^{j_2} \,dt .
	$$ 
	Using the identity  $\int_0^1  t^{j_1} (1-t)^{j_2} \,dt = 
	\frac{j_1! j_2!}{(n+1)!}$, we get 
	$$\sum_{i=1}^{m} \frac{i^{j_1}  (m-i)^{j_2}}{j_1!  j_2!} = \frac{m^{n+1}}{(n+1)!} +o(m^{n+1}).$$
	Using our previous considerations, we get
	\begin{equation*}
	\begin{split}
	\limsup_{m\to 0} \frac{n!}{m^{n+1}}  
	&	\length \left( \frac{
		H^0(\Xcal,  \Mcal + m\KL_1 )
	}{
		H^0(\Xcal, \Mcal + m\KL_2) 
	} \right) \\
	&		\leq 
	\frac{1}{n+1}\sum_{j_1+j_2=n} \int_{X^{\an}} \phi_{D} (dd^c \phi_{\Lcal_1} )^{j_1} \wedge (dd^c \phi_{\Lcal_2} )^{j_2}.
	\end{split}
	\end{equation*}
	By definition, the right hand side is $E(L, \phi_{\Lcal_1},\phi_{\Lcal_2})$. This proves \eqref{limsup inequality}.
	
	If we multiply \eqref{limsup inequality} by $-1$ and if we exchange $\Lcal_1$ with $\Lcal_2$, then we get the reverse inequality
	\begin{equation} \label{liminf inequality}
	\liminf_{m\to 0} \frac{n!}{m^{n+1}}  
	\length 
	\left( \frac {H^0(\Xcal,\Mcal + m{\KL_1})} {H^0(\Xcal,\Mcal + m {\KL_2})} \right)
	\geq 
	E(L, \phi_{\Lcal_1},\phi_{\Lcal_2}).
	\end{equation}	 
	Combining \eqref{limsup inequality} and \eqref{liminf inequality}, we get the claim in the proposition.
\end{proof}

\begin{theo}\label{cor3 energy}
	Let $L$ be {a} line bundle on a geometrically reduced 
	proper scheme $X$ over $K$ and let $\phi_1$ and $\phi_2$ be continuous semipositive metrics on $L^{\an}$. 
	Then we have
	\begin{equation}\label{vol-is-energy-main-thm}
	\vol(L,\phi_1,\phi_2) = E(L, \phi_1,\phi_2).
	\end{equation}
\end{theo}

\begin{proof}
	By base change and using Proposition \ref{cor vol comparison base change}  and Proposition  \ref{properties of energy}, we may assume that $K$ is algebraically closed.  By Chow's lemma 
	and birational invariance of non-Archimedean volumes and energy (see Lemma \ref{volumes and birationality} and Proposition \ref{properties of energy}), we may assume $X$ to be projective. 
	
	We first prove the claim for semipositive model metrics. Since the projective models of $X$ which are integrally closed in $X$ are cofinal by  Lemma \ref{lemma cofinal}, we may assume 
	that $\phi_i = \phi_{\Lcal_i}$ for some nef $\Q$-line bundle $\Lcal_i$ on some common projective model $\Xcal$ which is integrally closed in $X$. Homogeneity of the non-Archimedean volume   in Corollary \ref{lemma homogeneity} and the energy in Proposition \ref{properties of energy}  show that we may assume that the $\Lcal_i$ are nef line bundles on $\Xcal$. 
	Then the claim follows from Proposition \ref{prop energy 2'}, Lemma \ref{global sections and widehat} and Remark \ref{compare volumes}.
	
	Arbitrary continuous semipositive metrics  on $\Lan$ are uniform limits of semipositive model metrics on $\Lan$. 
	Then the claim follows from the first case as both the non-Archimedean volume 
	and 
	the energy are continuous in $( \phi_1, \phi_2)$ (see Propositions \ref{properties of energy} and \ref{basic properties of non-archimedean volume}).
\end{proof}

\section{Differentiability of non-Archimedean volumes}
\label{section inequality}

We will prove our main result about differentiation of non-Archimedean volumes. It generalizes \cite[Theorem B]{BGJKM} from the case of  discrete valuations to  arbitrary non-Archimedean complete absolute values and \cite[Theorem A]{BGM} from the ample to the nef case.

\subsection{Intermediate result for models}
\label{subsection model results}


We consider an $n$-dimensional projective scheme $X$ over $K$ with a projective model $\KX$ over $\Ko$. Let $D$ be a vertical Cartier divisor on $\Xcal$.   
Since $\KX$ is projective, 
we can write $\Ocal(D) = \KM_1 - \KM_2$ for nef line bundles $\KM_1,\KM_2$ on $\KX$. 
We consider a nef line bundle $\Lcal$ and  an arbitrary line bundle $\Ncal$ on $\KX$. 

\begin{lemma}
	\label{lemma inductive step}
	Unter the above assumptions, let $\phi_D$ be the model function associated to the vertical Cartier divisor $D$ and let
	$\Fcal_{j,m} \coloneqq \Ncal + m \Lcal +j(\KM_1-\KM_2)$. If $D$ is effective, then   we have
	\begin{equation}
	\label{eq inequality 2 model}
	\frac{n!}{m^n} {\length} \left( \frac {H^0(\KX,\Fcal_{j+1,m} )} {H^0(\KX,\Fcal_{j,m} ) } \right)\leq \int_\Xan \phi_D \, ( dd^c \phi_\Lcal  +dd^c \phi_{\KM_1}  )^{n} + o(1)
	\end{equation}
	for integers   $m \to \infty$ and all $j\in\{0,\ldots,m-1\}$. If $-D$ is effective, then $\geq$ holds in \eqref{eq inequality 2 model}. 
\end{lemma}

\begin{proof}
	We prove first the claim in the case $E \coloneqq -D$ effective.
	The canonical section $s_E$ of $\Ocal(E)=\KM_2-\KM_1$ determines a short exact sequence of coherent sheaves on $\KX$:
	\begin{equation*}
	0 \longrightarrow \Fcal_{j+1,m} \stackrel{\otimes s_E}{\longrightarrow}
	\Fcal_{j,m} \longrightarrow
	\Fcal_{j,m}|_{E} \longrightarrow 0
	\end{equation*}
	The start of the associated long exact sequence in cohomology is 
	\begin{equation*}
	0 \longrightarrow H^0( \KX, \Fcal_{j+1,m}) \stackrel{\otimes s_E}{\longrightarrow}
	H^0(\KX, \Fcal_{j,m} )  {\longrightarrow}
	H^0 (E, \Fcal_{j,m})  \longrightarrow \cdots 
	\end{equation*}
	and hence 
	\begin{equation} \label{estimate of the quotient}
	\length \left(
	H^0(\KX,\Fcal_{j,m}) /   H^0(\KX,\Fcal_{j+1,m}) 
	\right) \leq \length \left( H^0 (E, \Fcal_{j,m})  \right) =  h^0(E,\Fcal_{j,m})
	.
	\end{equation}  
	Using that $\KM_1$ and $\KM_2$ are {nef}, we deduce from Lemma  \ref{lemma nef increasing} that
	$$\vol(E,\Fcal_{j,m}) \leq  \vol(E,\Ncal + m\Lcal+j\KM_1) \leq  \vol(E, \Ncal+ m(\Lcal+\KM_1)). $$
	 These inequalities and Proposition \ref{asymptotic length formula for nef} give
	\begin{equation} \label{volume formula in the proof}
	h^0(E,\Fcal_{j,m}) \leq  \frac{m^n}{n!} \int_\Xan  \phi_E \, ( dd^c \phi_\Lcal  + dd^c \phi_{\KM_1} )^{n} + o(m^n).
	\end{equation}
	By \eqref{estimate of the quotient}, we get
	$$ \frac{n!}{m^n} \length \left(
	H^0(\KX,\Fcal_{j,m}) /   H^0(\KX,\Fcal_{j+1,m}) 
	\right) \leq  \int_\Xan  \phi_E \, ( dd^c \phi_\Lcal  + dd^c \phi_{\KM_1} )^{ n} + o(1).$$
	Using that $\phi_D = - \phi_E$, we get the desired reverse inequality in \eqref{eq inequality 2 model}.
	
	Now we deal with the  case $E \coloneqq D$ effective. The proof is quite similar as in the first case. The canonical global section $s_E$ induces a short exact sequence
	\begin{equation*}
	0 \longrightarrow \Fcal_{j,m} \stackrel{\otimes s_E}{\longrightarrow}
	\Fcal_{j+1,m} \longrightarrow
	\Fcal_{j+1,m}|_{E} \longrightarrow 0
	\end{equation*}
	of coherent sheaves on $\Xcal$. The same argument with the long exact cohomology sequence gives
	\begin{equation} \label{estimate of the quotient'}
	\length \left(
	H^0(\KX,\Fcal_{j+1,m}) /   H^0(\KX,\Fcal_{j,m}) 
	\right) \leq   h^0(E,\Fcal_{j+1,m}) 
	.
	\end{equation}
	The asymptotic formula \eqref{volume formula in the proof} holds still for $j+1$ instead of $j$ and so we get \eqref{eq inequality 2 model} by using \eqref{estimate of the quotient'}.
\end{proof}

\subsection{Main result}
\label{subsection main result}

In this subsection,  we assume that $X$ is an $n$-dimensional geometrically reduced proper scheme over $K$. We apply first the previous result to model functions.

\begin{lemma}
	\label{main prop}
	Let $L$ be a line bundle on $X$,  
	$f$ a model function on $X^{\an}$ and $\phi$ a continuous semipositive metric on $L$.
	Let $M$ be a line bundle on $X$ and $\phi_1, \phi_2$ 
	continuous semipositive metrics of $M$ such that 
	$f = \phi_1- \phi_2$.
	If $f\geq 0$, then 
	\begin{equation}
	\label{eq inequality 1}
	\vol(L, \phi +f , \phi) \leq \int_\Xan f \, ( dd^c \phi + dd^c \phi_1 )^{n}.
	\end{equation}
	If $f \leq 0$, then $\geq$ holds in \eqref{eq inequality 1}.
\end{lemma}

\begin{proof}
	By Corollary \ref{cor vol comparison base change}, 
	the left hand side is invariant under base change.
	By Proposition \ref{MA measures}(f), the right hand side is invariant under base change. 
	Hence we can assume that $K$ is algebraically closed.
	By Chow's lemma and birational invariance of non-Archimedean volumes and energy (see Lemma \ref{volumes and birationality} and Proposition \ref{properties of energy}), 
	we may assume $X$ projective.
	By continuity of the non-Archimedean volume and of the non-Archimedean Monge--Amp\`ere measures, we may assume that  
	$\phi, \phi_1,\phi_2$ are 
	induced by  nef $\Q$-line bundles $\Lcal, \Mcal_1,\Mcal_2$.  
	Using Lemma \ref{lemma cofinal}, we may assume that the $\Q$-line bundles are determined on a common projective model $\Xcal$ which is integrally closed in $X$. By the homogenity of the non-Archimedean volume, we may assume that $\Lcal,\Mcal_1,\Mcal_2$ are honest line bundles on $\Xcal$.

	Let us assume first that $f \geq 0$. By our above assumptions, we have a vertical Cartier divisor $D$ on $\Xcal$ with $\phi_D=f$. Since $\Xcal$ is integrally closed, Lemma \ref{lemma effective} shows that $D$ is an effective Cartier divisor. 
	We have to prove that 
	\begin{equation}
	\label{eq inequality 21}
	\limsup_{m \to \infty} \frac{1}{m^{n+1}/n!} \vol \left( \metr_{m(
		\phi_\Lcal+\phi_D)}, 
	\metr_{m\phi_\Lcal} 
	\right)
	\leq \int_\Xan \phi_D \, ( dd^c \phi_\Lcal  + dd^c \phi_{\Mcal_1} )^{ n}.
	\end{equation}
	By Remark \ref{compare volumes} and Lemma \ref{global sections and widehat}, 
	it is enough to prove that 
	\begin{equation*}
	\label{eq inequality 3}
	\limsup_{m \to \infty} \frac{1}{m^{n+1}/n!} \length \left( \frac{
		H^0(\Xcal, m (\Lcal + \KM_1- \KM_2) )
	}{
		H^0(\Xcal, m \Lcal)
	}  \right)
	\leq \int_\Xan \phi_D \, ( dd^c \phi_{\Lcal}  +dd^c \phi_{\KM_1}  )^{ n}.
	\end{equation*}
	We apply now Lemma \ref{lemma inductive step} with $\Ncal=\Ocal_\Xcal, \Lcal$ and $\Ocal(D)=\KM_1-\KM_2$ for the summands in 
	\begin{equation*} \label{addivity of length argument}
	{\length}  \left( \frac{
		H^0(\Xcal, m (\Lcal + \KM_1- \KM_2) )
	}{
		H^0(\Xcal, m\Lcal)
	}  \right) 
	= \sum_{j=0}^{m-1} 
	{\length} \left( \frac {H^0(\KX,\Fcal_{j+1,m} )} {H^0(\KX,\Fcal_{j,m} ) } \right).
	\end{equation*}
	This gives \eqref{eq inequality 21} and  proves the case $f \geq 0$. 
	Now assume that $f \leq 0$. Then the reverse inequality in \eqref{eq inequality 1} follows by applying the first case for $-f$ switching the role of $\Lcal_1,\Lcal_2$.
\end{proof}

\begin{theo}
	\label{corollary differentiablity yuan's argument}
	Let $X$ be an $n$-dimensional geometrically reduced proper scheme over $K$.
	Let $L$ be a line bundle on $X$, 
	$f$ a continuous real function on $\Xan$ and $\phi$ a continuous semipositive metric on $L$.
	Then 
	$$\frac{d}{dt}\bigg|_{t=0} \vol(L, \phi + tf, \phi) = \int_\Xan f \, (dd^c \phi)^{n}.$$
\end{theo}
\begin{proof}
	We have seen in  \ref{model metrics} that model functions are dense in the space of continuous functions on $\Xan$. By continuity and monoticity of non-Archimedean volumes in Proposition \ref{basic properties of non-archimedean volume}, we may assume that $f$ is a model function (see the proofs of \cite[Theorem 5.4.3]{BGJKM} or \cite[Theorem 3.1]{BGM} for details). 
	It follows from \ref{model metrics} that $L$ is nef and hence $\vol(L)=\deg_L(X)$. 
	If $f\leq 0$, then  Lemma \ref{main prop} implies
	\begin{equation} \label{lower bound for na volume}
	\liminf_{t\to 0+} \frac{\vol(L, \phi + tf, \phi)}{t} 
	\geq  \int_\Xan f \, (dd^c \phi)^{n}.
	\end{equation}
	Using that for $c\in \R$, we have
	$\vol (L,\phi+c , \phi) = \vol(L)+c\deg_L(X)$ (see Proposition \ref{basic properties of non-archimedean volume}), 
	we deduce from Proposition \ref{MA measures}(e) that \eqref{lower bound for na volume} holds for any continuous function $f$.
	Using Lemma \ref{main prop} in the case $f \geq 0$, 
	a similar trick shows that
	\begin{equation*} \label{lower bound for na volume'}
	\limsup_{t\to 0+} \frac{\vol(L, \phi + tf, \phi)}{t} 
	\leq  \int_\Xan f \, (dd^c \phi)^{n}.
	\end{equation*} 
	holds for any continuous function $f$ and hence we get 
	$$\lim_{t\to 0+} \frac{\vol(L, \phi + ft, \phi)}{t}
	=
	\int_\Xan f \, (dd^c \phi)^{n}.
	$$
	Finally replacing $f$ by $-f$, we get the same for $t \to 0-$ proving the claim.
\end{proof}

\section{Orthogonality property} \label{Section: differentiability of energy}

In this section, we prove the orthogonality property of a continuous semipositive metric $\phi$ assuming that the semipositive envelope $\Env(\phi)$ is continuous.

\subsection{Semipositive envelope} \label{subsection differentiability in the semipositve case}

Let $L$ be a line bundle on a proper scheme $X$ over $K$. We will introduce the semipositive envelope $\Env(\psi)$ of a bounded metric $\psi$ on $\Lan$. We always assume that $L$ has at least one semipositive model metric. This implies that $L$ is nef (see \ref{model metrics}) and holds at least  for semiample line bundles.

Recall from \ref{notation metric} that we use additive notation for metrics and hence $\phi \leq \psi$ for metrics on $\Lan$ means $\abs_\phi \geq \abs_\psi$ for the corresponding norms on fibers. 

\begin{defi} \label{psh envelope}
	The {\it semipositive envelope} $\Env(\psi)$ of a bounded metric $\psi$ on $\Lan$ is defined by
	$$
	\Env(\psi) := \sup \{ \phi \st \phi \text{ is a continuous semipositive metric on $\Lan$ and } \phi \leq \psi\}.
	$$ 	
\end{defi}

\begin{rem} \label{psh continuity and semipositive}
	Note that in the above definition, we may restrict our attention to semipositive model metrics of $L$ using that every semipositive continuous metric is a uniform limit of continuous semipositive metrics. It follows from the assumed existence of a semipositive model metric that $\Env(\psi)$ is a bounded metric on $L$.

	If $\Env(\psi)$ is a continuous metric on $\Lan$, then it follows from Dini's theorem that $\Env(\psi)$ is a continuous semipositive metric on $\Lan$. If $\psi$ is a continuous metric, continuity of $\Env(\psi)$ is not clear. This property is expected in case of semiample line bundles on a normal projective variety (see \cite[Conjecture 7.31]{BE}).

	We refer to \cite[\S 5.3, 5.4]{BJ18} and \cite[\S 7.5]{BE} for the study of the \emph{psh-envelope} which agrees with the semipositive envelope at least in the case of a continuous metric on an ample line bundle. For a continuous metric $\psi$ of $L$, the definition of $\Env(\psi)$ agrees with  \cite[Definition 2.5.1]{BGJKM} where  multiplicative notation for metrics was used.
\end{rem}


\begin{prop} \label{easy properties of envelope}
	Let $\psi, \psi_1, \psi_2$ be bounded metrics on $\Lan$. Then we have:
	\begin{itemize} 
		\item[(a)] $\Env(\psi) \leq \psi$ with equality if $\psi$ is a  continuous semipositive metric.
		\item[(b)] If $a \in \N$, then $\Env(a\psi)= a \Env(\psi)$.
		\item[(c)] If $\psi_1 \leq \psi_2$, then $\Env(\psi_1) \leq \Env(\psi_2)$.
		\item[(d)] If $c \in \R$, then $\Env(\psi + c) = \Env(\psi)+c$.
		\item[(e)] $\sup_\Xan|\Env(\psi_1)-\Env(\psi_2)|\leq \sup_{\Xan}|\psi_1-\psi_2|.$
	\end{itemize}	
\end{prop}

\begin{proof}
	Properties (a)--(d) are obvious from the definition and (e) follows from (b)--(d).	
\end{proof}

We generalize now \cite[Proposition 6.2.1]{BGJKM} and \cite[Proposition 9.11(vi)]{BE}.

\begin{prop} \label{sup norm and envelope}
	Let $\psi$ be bounded metric on $L^{\an}$.
	\begin{itemize}
		\item[(a)]
		We have $\metr_{\Env(\psi)}= \metr_\psi$ for the corresponding supremum seminorms on $H^0(X,L)$. 
		\item[(b)]
		If $X$ is reduced, then we get  $\vol(L,\phi, \psi) = \vol (L, \Env(\phi), \Env(\psi) )$ 
		for any bounded metric  $\phi$ on $\Lan$.
	\end{itemize}
\end{prop}

\begin{proof}
	By Proposition \ref{easy properties of envelope}(a), we deduce $\metr_{\Env(\psi)} \geq \metr_\psi$. We prove the converse inequality by contradiction. Assume that  there is $s \in H^0(X,L)$ and $x \in \Xan$ with $|s(x)|_{\Env(\psi)} > ||s||_\psi$. Rescaling the metric $\psi$ and using Proposition \ref{easy properties of envelope}(d), we may assume
	\begin{equation} \label{rescaled assumptions}
	|s(x)|_{\Env(\psi)} >1 = ||s||_\psi.
	\end{equation}
	This yields $g \coloneqq \psi \circ s \geq 0$ and hence we get a singular metric $\psi_{s} \coloneqq \psi - g \leq  \psi$ on $\Lan$. We note that the results of \cite[\S 6.1]{BGJKM} hold over any (non-trivially valued) non-Archimedean field, as the crucial reference \cite{gubler-martin} works in this setting. For a continuous semipositive metric $\psi_1 \leq \psi$ on $\Lan$, it follows from \cite[Lemma 6.1.3]{BGJKM} that $\psi' \coloneqq \max(\psi_1, \psi_{s})$ is a continuous semipositive metric on $\Lan$ with 
	$\psi_1 \leq \psi' \leq \psi$
	and hence 
	\begin{equation}  \label{lower bound by envelope}
	\psi' \leq \Env(\psi)
	\end{equation}
	by definition of the semipositive envelope.  By construction, we have $\psi_s \circ s \equiv  0$ and hence \eqref{lower bound by envelope} yields
	$$\Env(\psi) \circ s  \geq \psi' \circ s = \max(\psi_1\circ s, \psi_{s}\circ s) \geq 0$$
	which contradicts the strict inequality in \eqref{rescaled assumptions}. This proves $\metr_{\Env(\psi)}= \metr_\psi$. 
	
	Finally, (b) follows from the definition of the non-Archimedean volume and by applying (a) to the bounded metrics $m\psi$ and to $m\phi$ for any $m \in \N$. 
\end{proof}

The following generalizes \cite[Corollary 6.2.2]{BGJKM} and \cite[Corollary 9.16]{BE}.

\begin{cor} \label{lemma volume energy}
	Let $L$ be a line bundle over a geometrically reduced proper scheme $X$ over $K$. Assume that $\psi, \phi$ are continuous  metrics on  $\Lan$ and assume that $P(\psi),P(\phi)$ are also continuous metrics on $\Lan$. Then we have
	$$\vol(L,\phi,\psi) = \en(L,\Env(\phi),\Env(\psi)).$$
\end{cor}

\begin{proof}
	This follows by combining Proposition \ref{sup norm and envelope} and Theorem \ref{cor3 energy}.	
\end{proof}

\subsection{Orthogonality}  In this subsection, we consider a line bundle $L$ on a proper scheme $X$ of dimension $n$ over $K$. 
\label{subsection orthogonality}

\begin{defi}
	Let  $\phi$ be a continuous metric on $\Lan$  with $\Env(\phi)$ a continuous metric on $\Lan$ as well. We say the $\phi$ satisfies the {\it orthogonality property} if 
	\begin{equation} \label{orthogonality property}
	\int_\Xan (\Env(\phi) - \phi) \,(dd^c\Env(\phi))^n = 0.
	\end{equation}
\end{defi}

\begin{theo}
	\label{lemma orthogonality}
	We assume that $X$ is geometrically reduced. Let  $\phi$ be a continuous metric on $L^{\an}$ such that $\Env(\phi)$ is also a continuous metric. Then $\phi$ satisfies the orthogonality property.
\end{theo}
\begin{proof}
	This follows from Theorem \ref{corollary differentiablity yuan's argument} and Proposition \ref{sup norm and envelope} by the same arguments	as in the proof of \cite[Theorem~6.3.2]{BGJKM}.
\end{proof}

\begin{rem} \label{differentiability of energy}
	Assume that the semipositive envelope of any continuous metric $\phi$ on $L$ is continuous.  Fixing a continuous reference metric $\psi$ of $L$,  orthogonality for all  continuous metrics $\phi$ of $L$ is equivalent to differentiability of $\phi \mapsto E(L,\Env(\phi),\Env(\psi))$ for all continuous metrics $\phi$ of $L$,
	see \cite[Lemma 3.5]{BGM} for the argument.  
	 This differentiability is a crucial property in the proof of the existence of solutions of non-Archimedean Monge--Amp\`ere equations, see \cite{BFJ2}.	
\end{rem}

\appendix

\renewcommand\thesection{\Alph{section}}
\setcounter{section}{0}

\section*{Appendix A. Asymptotical functions} \label{appendix on asymptotical functions}
\addcontentsline{toc}{section}{Appendix A. Asymptotical functions}\refstepcounter{section}
In what follows, $P$ denotes an arbitrary abelian group, and we set as usual $P_\R \coloneqq P \otimes_\Z \R$. We use the norm $|x| \coloneqq \sum_i |x_i|$ for $x \in \R^r$.

The goal of this appendix is to establish the following elementary result. 
\begin{theo} \label{thm:hh} Let $h:P\to\R$,  $s \in \R_{>0}$, and assume that for all $L_0,\dots,L_r\in P$ and $m=(m_1,\dots,m_r)\in\Z^r$ we have 
	\begin{equation} \label{asymptotic growth condition}
	h(m_1L_1+ \dots + m_r L_r)= O(|m|^s)  
	\end{equation}
	and 
	\begin{equation} \label{asymptotic continuity condition}
	h(L_0+m_1L_1+ \dots + m_r L_r)- 
	h(m_1L_1+ \dots + m_r L_r)= O(|m|^{s-1}). 
	\end{equation}
	Then there is  a unique function $\hh\colon P_\R\to\R$ such that: 
	\begin{itemize}
		\item[(i)] for any $L \in P$ we have 
		\begin{equation}\label{equ:hh}\hh(L)=\limsup_{m\to+\infty} m^{-s}h(mL); 
		\end{equation}
		\item[(ii)] $\hh$ is homogeneous of degree $s$, \textit{i.e.} $\hh(tM)=t^s\hh(M)$ for $t \in \R_+$ and $M \in P_\R$; 
		\item[(iii)] $\hh$ is continuous on any finite dimensional real subspace of $P_\R$.
	\end{itemize}
\end{theo}
To simplify the notation, we slightly abusively denote by $L\in P_\R$ the image of $L\in P$. 
The above abstract setting is inspired by the following example.
\begin{example} \label{higher cohomological functions}
	Let $Y$ be a projective scheme over a field $k$ of dimension $n$. For $q \in \{0, \dots , n\}$, the function $L\mapsto h^q(Y,L)$ on $\Pic(Y)$ satisfies~\eqref{asymptotic growth condition} and~\eqref{asymptotic continuity condition} with $s=n$. Up to a factor $n!$, the induced function $\hh$ is then K\"uronya's higher cohomological function $\hh^q:\Pic(Y)_\R\to\R_{\ge 0}$, which coincides with the volume of line bundles for $q=0$. We refer to~\cite[\S 3.4]{BGJKM} for details. 
\end{example}

In this paper, we will apply the appendix to the following setting.
\begin{example} \label{higher cohomological functions+}
	Let $Y$ be an $n$-dimensional finitely presented projective torsion scheme  over $\Ko$ for a non-Archimedean field $K$. For $q \in \{0, \dots , n\}$, we define $h^q(Y,L)\in\R_{\ge 0}$ as the content of the torsion module $H^q(Y,L)$, \emph{cf.}~Definition~\ref{definition Euler characteristic}. It is shown in Lemma~\ref{lemma asymptotic easy} that this function satisfies~\eqref{asymptotic growth condition} and~\eqref{asymptotic continuity condition} with $s=n$, and therefore induces asymptotic cohomological functions $\hh^q\colon\Pic(Y)_\R\to\R_{\ge 0}$. As in Example~\ref{higher cohomological functions}, we will normalize $\hh^q$ in this case with a factor $n!$.
\end{example}

\begin{lemma} \label{lem:perturbation}
	For any $L_1, \dots, L_r\in P$, there exists $C>0$ such that 
	$$
	\left|h(m_1L_1+ \dots + m_rL_r)-h(m'_1L_1+\dots+m'_r L_r)\right|\le C|m-m'|\max\{|m|,|m'|\}^{s-1}
	$$
	for all $m,m'\in\Z^r$.  
\end{lemma}

\begin{proof}
	By \eqref{asymptotic continuity condition} we can find $C>0$ such that 
	$$
	\left| h(m_1L_1+ \dots + m_rL_r\pm L_i)-h(m_1 L_1+\dots+m_r L_r)\right|\le C|m|^{s-1}. 
	$$
	for all $m\in\Z^r$ and $i=1,\dots,r$. An iterated application of this estimate yields the result. 
\end{proof}

\begin{proof}[Proof of Theorem~\ref{thm:hh}] Without loss of generality, we may assume that $P$ is finitely generated. Uniqueness is then clear, since (i) and (ii) uniquely determine $\hh$ on $P_\Q$, which is dense in the finite dimensional vector space $P_\R$.

	In a first step we show that $\tilde h:P \to \R$ defined by 
	$$
	\tilde h(L)  \coloneqq \limsup_{m\to+\infty} m^{-s}h(mL)
	$$
	is $\N$-homogeneous of degree $s$, \textit{i.e.} 
	\begin{equation}\label{equ:homog}
	\tilde h(aL) = a^s \tilde h(L)
	\end{equation}
	for $a\in \N$ and $L\in P$. 
     Note first that $\tilde h(L)$ is real-valued by \eqref{asymptotic growth condition}.
	The case $a=0$ follows from $s>0$, and so we may assume $a\ge 1$. We obviously have
	$$
	\tilde h(L)  \ge \limsup_{m\to +\infty} (am)^{-s} h(amL)=a^{-s}\tilde h(aL)
	$$
	Conversely, pick a sequence $m_j\to+\infty$ such that $m_j^{-s}h(m_jL)\to\tilde h(L)$, and write $m_j=a q_j+r_j$ with $q_j\in\N$ and $r_j\in\{0,\dots,a-1\}$. Since $r_j$ takes only finitely many values, Lemma \ref{lem:perturbation} yields 
	$$
	h(m_jL)=h(aq_j L)+O(m_j^{s-1}), 
	$$
	and hence 
	$$
	a^s\tilde{h}(L) = a^s \lim_j m_j^{-s}h(m_jL)= \lim_j q_j^{-s}h(aq_jL)\leq \tilde{h}(aL),
	$$
	which proves~\eqref{equ:homog}. 
	
	This first step yields that there is a unique function $\hh:P_\Q\to\R$ which is $\Q_+$-homogeneous of degree $s$ and which satisfies~\eqref{equ:hh}. It remains to show that $\hh$ extends continuously to $P_\R$. 
	Pick $L_1,\dots,L_r\in P$. By Lemma~\ref{lem:perturbation}, we have 
	$$
	\left|h(m_1L_1+ \dots + m_rL_r)-h(m'_1L_1+\dots+m'_r L_r)\right|\le C|m-m'|\max\{|m|,|m'|\}^{s-1},
	$$
	for all $m,m'\in\Z^r$. By homogeneity, this yields
	$$
	\left|\hh(x_1L_1+ \dots + x_rL_r)-\hh(x'_1L_1+\dots+x'_r L_r)\right|\le C|x-x'|\max\{|x|,|x'|\}^{s-1}
	$$
	for $x,x'\in\Q^r$. As a result, $\hh:P_\Q\to\R$ is uniformly continuous on each bounded subset of $P_\Q$, and hence admits a unique continuous extension to $P_\R$. 
\end{proof}

\begin{prop}\label{prop:hh} In the setting of Theorem~\ref{thm:hh}, pick $L_1,\dots,L_r\in P$, and set for $x\in\R^r$ and $m \in \N$ 
	$$
	f_m(x):=m^{-s}h\left(\sum_i\lfloor m x_i\rfloor L_i\right).
	$$
	\begin{itemize}
		\item[(i)] For all $x\in\R^r$, we have $\hh(\sum_i x_i L_i)=\limsup_{m\to +\infty} f_m(x)$. 
		\item[(ii)] Assume given an open convex cone $\sigma \subset \R^r$ such that $\lim_{m\to +\infty} f_m(x)=\hh(x)$ for all $x\in \sigma\cap\Z^r$. Then $\lim_{m\to +\infty} f_m(x)=\hh(x)$ for all $x\in\overline{\sigma}$. 
	\end{itemize}
\end{prop}
\begin{proof} To prove (i), we set  $\hh(x):=\hh(\sum_i x_i L_i)$ and $f(x):=\limsup_{m \to\to +\infty} f_m(x)$ for $x\in\R^r$. We note that $f(x)=\hh(x)$ for $x\in\Z^r$, by~\eqref{equ:hh}. 
	
	By Lemma~\ref{lem:perturbation},  for all $m\in\Z_{>0}$ and  $x,x'\in \R^r$ with $|x|,|x'|\leq R$ we have a uniform estimate  
	\begin{equation}\label{equ:flip}
	\left|f_m(x)-f_m(x')\right|\le C |x-x'|\max\{|x|,|x'|\}^{s-1}+ O_R(m^{-1}). 
	\end{equation}
	This yields $|f(x)-f(x')|\le C|x-x'|\max\{|x|,|x'|\}^{s-1}$, which shows that $f$ is continuous on $\R^r$. Arguing just as for~\eqref{equ:homog}, we further have $f(ax)=a^s f(x)$ for all $a\in\Z_{>0}$ and $x\in\R^r$. It follows that $f=\hh$ on $\Q^r$, and hence also on $\R^r$, by continuity. This proves (i).
	
	To prove (ii), we show in a first step that 
	$$
	A:=\{x\in\R^r\mid \lim_{m \to \infty} f_m(x)=\hh(x)\}
	$$ 
	is a closed subset of $\R^r$. Let us pick a sequence $(x_n)_{n \in \N}$ in $A$ converging to $x \in \R^r$. For $\varepsilon > 0$, continuity of $\hh$ yields that 
	$|\hh(x_n)-\hh(x)|<\varepsilon/3$ for $n \gg 0$. We choose such an $n$ which also satisfies 
	$$C |x-x_n|\max\{|x|,|x_n|\}^{s-1}<\varepsilon/6.$$
	There is $m_0 \in \N$ such that the term ${O_R}(m^{-1})$ for $x'=x_n$ in \eqref{equ:flip} is bounded by $\varepsilon/6$ for all $m \geq m_0$ and all $n\in \N$. It follows from \eqref{equ:flip}  that
	$|f_m(x_n)-f_m(x)|< \varepsilon/3$ for $m \geq m_0$. Since $x_n\in A$, there is $m(n)\in \N_{\geq m_0}$ depending on $n$ such that $|f_m(x_n)-\hh(x_n)|<\varepsilon/3$ for all $m \geq m(n)$. Overall, the triangle inequality gives $|f_m(x)-\hh(x)|<\varepsilon$ for all $m \geq m(n)$. This proves $\lim_{m \to \infty} f_m(x)=\hh(x)$ and hence $x \in A$. We conclude   that $A$ is closed.
	
	By  assumption, $A$ contains $\sigma\cap\Z^r$. By the first step, to prove (ii) it will thus be enough to show that $A$ contains $\sigma\cap\Q^r$, which is dense in the closed convex cone $\overline{\sigma}$. Let $x\in \sigma\cap\Q^r$, and pick $a\in\Z_{>0}$ such that $ax\in\Z^r$. For $m\in\Z_{>0}$, write $m=aq_m+r_m$ with $q_m\in\N$ and $r_m\in\{0,\dots a-1\}$. Then $\lfloor m x_i\rfloor-q_m ax_i=\lfloor r_m x_i\rfloor$ remains bounded, and Lemma~\ref{lem:perturbation} thus yields a constant $C' >0$ depending on $|x|$ such that
	\begin{equation}\label{equ:hhom}
	\left|h(\sum_i\lfloor m x_i\rfloor L_i)-h(q_m \sum_i a x_i L_i)\right|\le C' m^{s-1}
	\end{equation}
	for all $m \in \Z_{>0}$. Since $ax\in \sigma\cap\Z^r$, we have by assumption $q_m^{-s} h(q_m \sum_i ax_i L_i)\to\hh(ax)=a^s\hh(x)$. Using~\eqref{equ:hhom}, we get $m^{-s} f_m(x)\to\hh(x)$, \textit{i.e.} $x\in A$, which concludes the proof of (ii).
\end{proof}


\newcommand{\etalchar}[1]{$^{#1}$}
\def\cprime{$'$}
\ifx\undefined\bysame
\newcommand{\bysame}{\leavevmode\hbox to3em{\hrulefill}\,}
\fi

\end{document}